\documentclass{amsart}

\usepackage{amsmath,amsfonts,amssymb}
\usepackage{hyperref}
\usepackage{geometry}
\usepackage{enumerate}
\usepackage{graphicx}
\usepackage{subcaption}

\newtheorem{thm}{Theorem}[section]
\newtheorem{lem}[thm]{Lemma}
\newtheorem{prop}[thm]{Proposition}
\newtheorem{coro}[thm]{Corollary}
\newtheorem{Def}[thm]{Definition}

\newtheorem{rem}[thm]{Remark}

\usepackage{tikz}
\usepackage{pgfplots}
\usepackage{pgfplotstable}
\usetikzlibrary{calc}
\usepackage{tkz-euclide}
\usepackage{tkz-base}
\usetkzobj{all}

\tikzset{courbe/.append style={scale=0.7}}

\def\ds{\displaystyle}

\def\R{\mathbb{R}}

\def\be{\begin{equation}}
\def\ee{\end{equation}}

\def\n{{\boldsymbol n}}
\def\p{\partial}
\def\grad{\boldsymbol{\nabla}}
\def\div{\grad\cdot}

\def\O{\Omega}

\def\a{\alpha}
\def\eps{\epsilon}
\def\x{{\boldsymbol x}}
\def\bt{{\boldsymbol t}}
\def\d{{\rm d}}
\def\ov#1{\overline{#1}}

\def\wh#1{\widehat{#1}}
\def\wt#1{\widetilde{#1}}

\def\bphi{\boldsymbol{\phi}}
\def\bvarphi{\boldsymbol{\varphi}}

\def\bJ{{\boldsymbol  J}}

\def\bV{{\boldsymbol  V}}
\def\bW{{\boldsymbol  W}}

\def\bv{{\boldsymbol v}}
\def\by{{\boldsymbol y}}

\def\bw{{\boldsymbol w}}
\def\0{{\bf 0}}
\def\1{{\bf 1}}

\def\Aa{\mathcal{A}}
\def\bAa{\boldsymbol{\Aa}}

\def\Dd{\mathcal{D}}
\def\Ee{\mathcal{E}}
\def\Ff{\mathcal{F}}

\def\Hh{\mathcal{H}}

\def\Xx{\mathcal{X}}
\def\bXx{{\boldsymbol{\Xx}}}

\def\Zz{\mathcal{Z}}
\def\bZz{\boldsymbol{\mathcal{Z}}}

\def\Ddd{\mathfrak{D}}

\def\dt{{\Delta t}}
\def\bc{{\boldsymbol{c}}}

\def\bmu{{\boldsymbol{\mu}}}

\begin{document}

\title[A degenerate Cahn-Hilliard model as constrained Wasserstein gradient flow]{A two-phase two-fluxes degenerate Cahn-Hilliard model as constrained Wasserstein gradient flow}
\author{Cl\'ement Canc\`es}
\address{Inria, Univ. Lille, CNRS, UMR 8524 - Laboratoire Paul Painlev\'e, F-59000 Lille (\href{mailto:clement.cances@inria.fr}{\tt clement.cances@inria.fr})}
\author{Daniel Matthes}
\address{Zentrum f\"ur Mathematik, Technische Universit\"at M\"unchen, 85747 Garching, Germany 
(\href{mailto:matthes@ma.tum.de}{\tt matthes@ma.tum.de})}
\author{Flore Nabet}
\address{CMAP, Centre de Math\'ematiques Appliqu\'ees, \'Ecole Polytechnique, Route de Saclay, 91128 Palaiseau Cedex, France (\href{mailto:flore.nabet@polytechnique.edu}{\tt flore.nabet@polytechnique.edu})}

\begin{abstract}
        We study a non-local version of the Cahn-Hilliard dynamics for phase separation 
in a two-component incompressible and immiscible mixture with linear mobilities. 
   In difference to the celebrated local model with nonlinear mobility,
   it is only assumed that the divergences of the two fluxes --- but not necessarily the fluxes themselves --- annihilate each other.
   Our main result is a rigorous proof of existence of weak solutions.
   The starting point is the formal representation of the dynamics as a constrained gradient flow in the Wasserstein metric.
   We then show that time-discrete approximations by means of the incremental minimizing movement scheme
   converge to a weak solution in the limit.
   Further, we compare the non-local model to the classical Cahn-Hilliard model in numerical experiments.
   Our results illustrate the significant speed-up in the decay of the free energy due to the higher degree of freedom for the velocity fields.
\end{abstract}


\maketitle

\section{Motivation and presentation of the model}

\subsection{Introduction} 
 {We are interested in a non-local Cahn-Hilliard system
\begin{subequations}
  \label{eq:syst-PDE2}
  \begin{align}
    \ds \p_t c_i - \div\left( {m_i c_i} \grad \left(\mu_i + \Psi_i\right)\right) \label{eq:1a}
    &= m_i \theta_i \Delta c_i \quad \text{for}\; i \in \{1,2\},\\
    c_1 + c_2 &= 1,  \label{eq:1b} \\
    \mu_1 - \mu_2 &= -\alpha \Delta c_1 + \chi (1-2c_1),
  \end{align}  
\end{subequations}
modelling the flow of an incompressible and immiscible mixture in a bounded convex domain $\O$.
Above, $m_1$ and $m_2$ are positive mobility constants,
$\chi>0$ is a parameter for the chemical activity,
and the non-negative numbers $\theta_1$ and $\theta_2$ control the degree of thermal agitation.
The system is complemented with initial conditions
\be\label{eq:initial}
(c_i)_{|_{t=0}} = c_i^0 \in H^1(\O) \; \text{with}\; c_i^0 \ge 0 \; \text{and}\; c_1^0 + c_2^0 = 1 \; \text{in}\; \O, 
\ee
and boundary conditions
\be\label{eq:BC}
c_i \grad (\mu_i + \Psi_i) \cdot \n = \grad c_i\cdot \n= 0 \; \text{on}\; \p\O \times (0,\infty).
\ee}

 {
Equation~\eqref{eq:1a} can be rewritten as a conservation law 
\be\label{eq:cons_i_intro}
\p_t c_i + \div\bJ_i = 0, 
\ee
where
\[
\bJ_i = - {m_i c_i} \grad \left(\mu_i + \Psi_i + \theta_i \log(c_i)\right).
\]
Summing \eqref{eq:cons_i_intro} over $i\in \{1,2\}$ and using~\eqref{eq:1b} yields
\[
\div \bJ_{\rm tot} = 0 \quad\text{where}\quad \bJ_{\rm tot} = \bJ_1 + \bJ_2
\]
denotes the total flux. In particular, we do not require that  $\bJ_{\rm tot} = \0$ as for the 
classical (or local) degenerate Cahn-Hilliard model~\cite{deGennes80,Gurtin96,EG96}. 
The relaxation of the constraint from vanishing total flux
to divergence free total flux was initially proposed by E and Palffy-Muhoray in~\cite{EPM97} and 
studied formally by Otto and E in~\cite{OE97}. It was in particular noticed in~\cite{OE97}  that 
the system~\eqref{eq:syst-PDE2} can be interpreted as the constrained Wasserstein gradient flow 
of some Ginzburg-Landau energy where the velocity field $\bV = (\bv_1,\bv_2)$ transporting 
the concentrations $\bc = (c_1,c_2)$ has to preserve the constraint~\eqref{eq:1b}.
}

The nonlocal model is derived formally in Section~\ref{ssec:formal}, then compared to the classical (or local)
degenerate Cahn-Hilliard model in Section~\ref{ssec:comparison}. Numerical illustrations of its behavior
are given in Section~\ref{ssec:num}. In Section 2, we introduce the necessary material to prove our main result, 
that is the global existence of a weak solution to the model~\eqref{eq:syst-PDE2}--\eqref{eq:BC}. 
This existence result is obtained by showing the convergence of a minimizing movement scheme {\em \`a la} 
Jordan, Kinderlehrer, and Otto~\cite{JKO98}. Section 3 is devoted to the proof of the convergence of the minimizing 
movement scheme.

\subsection{Derivation of the model}\label{ssec:formal}

We consider an incompressible mixture composed of two phases flowing within 
a  {bounded} open convex subset $\O$ of $\R^d$ with $d \leq 3$.
The fluid is incompressible, so its composition at time $t\ge 0$ is fully 
described by the saturations $c_i(\x,t) \in [0,1]$, $i \in \{1,2\}$, i.e., the volume ratio 
of the phase $i$ in the fluid. This leads to the constraint 
\be\label{eq:c1+c2}
c_1  + c_2 = 1 \qquad \text{in}\; \O \times (0,\infty).
\ee
 {We assume that this constraint is already satisfied at the initial time $t=0$}, 
where the composition of the mixture is given by $\bc^0 = (c_1^0, c_2^0): \O \to [0,1]^2$, 
 {i.e.}
\be\label{eq:c10+c20}
c_1^0  + c_2^0 = 1 \qquad \text{in}\; \O.
\ee
 {We further assume that both phases have positive mass.}

The motions of each phase is governed by a linear transport 
equation
\be\label{eq:cons-i}
\p_t c_i + \div(c_i \bv_i) = 0 \quad \text{in } \O \times (0,\infty), 
\ee
where $\bv_i : \O \times \R_+ \to \R^d$ denotes the speed of the phase $i$, and $\bV = (\bv_1, \bv_2)$. 
 {To enforce mass conservation,
the boundary $\p\O$ of $\O$ is assumed to be impervious, hence}
\begin{align}
  \label{eq:imper}
  c_1\bv_1\cdot\n = c_2\bv_2\cdot\n = 0 \quad \text{on $\p\O$},
\end{align}
where $\n$ denotes the outward normal to $\p\O$,
 {so that Gauss' theorem implies}
\be\label{eq:m_i}
\int_\O c_i(\x,t) \d\x = \int_\O c_i^0(\x) \d\x = |\O| \ov c_i, \qquad t\ge 0, \; i \in \{1,2\}. 
\ee
Consequently,
at each time $t\ge0$, the saturations $\bc(t)$ belong to 
the set $\bAa = \Aa_1 \times \Aa_2$,
where
$$
\Aa_i = 
\left\{\text{ {$c_i:\O\to[0,1]$ measurable}}\; \middle| \; \int_\O c_i(\x) \d\x= |\O| \ov c_i \right\}, \qquad i \in \{1,2\}.
$$

To each configuration $\bc \in \bAa$, we associate the energy 
$$
\Ee(\bc) = \Ee_{\rm Dir}(\bc) + \Ee_{\rm chem}(\bc) + \Ee_{\rm therm}(\bc)+  \Ee_{\rm cons}(\bc) 
+ \Ee_{\rm ext}(\bc),
$$
 {whose components are as follows.}
\begin{itemize}
\item \textbf{Dirichlet energy:}
   {for given $\alpha_1,\alpha_2>0$},
  $$
  \Ee_{\rm Dir}(\bc) := \begin{cases} 
    \ds \sum_{i\in \{1,2\}} \frac{\alpha_i }2 \int_\O |\grad c_i|^2 \d\x & \text{if}\; \bc \in H^1(\O)^2, \\
    +\infty & \text{otherwise},
  \end{cases}
  $$
  penalizes the spatial variations of the saturation profiles. 
\item \textbf{Chemical energy:} 
   {for a given $\chi>0$},
  $$
  \Ee_{\rm chem}(\bc) = \chi \int_\O c_1 c_2 \d\x
  $$
  measures the impurity of the mixture.
\item \textbf{Thermal energy:}
   {for given $\theta_1,\theta_2\ge0$},
  $$
  \Ee_{\rm therm}(\bc) = \sum_{i\in\{1,2\}} {\theta_i} \int_\O H(c_i) \d\x, 
  \quad\text{with}\quad 
  H(c) = c \log(c) - c + 1 \ge 0, 
  $$
  is the free thermal energy.
  The case $\theta_1 = \theta_2 = 0$ is called the deep quench limit \cite{CENC96}. 
\item \textbf{Constraint:}
  $$
  \Ee_{\rm cons}(\bc) = \int_\O E_{\rm cons}(\bc) \d\x, \quad \text{where}\quad 
  E_{\rm cons}(\bc) = \begin{cases}
    0 & \text{if}\; c_1 + c_2 = 1,\\
    +\infty & \text{otherwise},
  \end{cases}
  $$
  realizes the constraint~\eqref{eq:c1+c2}.
\item \textbf{Exterior potential:}
   {for given external potentials $\Psi_1,\Psi_2\in  {H^{1}(\O)}$},
  $$
  \Ee_{\rm ext}(\bc) = \sum_{i\in\{1,2\}} \int_\O c_i(\x) \Psi_i(\x) \d\x
  $$
  is the potential energy related to exterior forces like gravity or electrostatic forces. 
\end{itemize}
All of these components are convex in $\bc$ \emph{except} for the chemical energy $\Ee_{\rm chem}$,
which, however, is smooth.
 {Denote by 
  $$\bXx = \left\{\bc\in H^1(\O; [0,1])^2 \,\middle| \, c_1 + c_2 = 1 \; \text{a.e. in}\; \O\right\}$$
  the domain of $\Ee$, i.e., 
  \be\label{eq:bXx}
  \Ee(\bc) < \infty \quad \Leftrightarrow \quad \bc \in \bXx. 
  \ee
}

 {Before entering into the rigorous derivation of the PDEs that govern the
gradient flow of the energy $\Ee(\bc)$ with respect to a tensorized Wasserstein distance,
we provide formal calculations based on}
the framework of generalized gradient flows of~\cite{Mie11, Pel-lecture, CGM15} 
in order to identify the underlying PDEs. 

The motion of the phases induces the viscous dissipation
\be\label{eq:Ddd}
\Ddd(\bc, \bV) = \sum_{i\in\{1,2\}}\frac1{2m_i} \int_\O c_i |\bv_i|^2 \d\x, 
\qquad \forall \bc \in \bAa, \; \forall  {\bV \in \bZz(\bc)}, 
\ee
where $m_i$ is the mobility coefficient of the $i$th phase, and where 
\[
\bZz(\bc) = \left\{ \bV=(\bv_1,\bv_2) : \O \to \left(\R^d\right)^2 \; \middle| \; c_i \bv_i \cdot \n = 0 \; \text{on}\; \p\O \right\} 
\]
denotes the space of the admissible (but unconstrained) vector fields.
 {This quantity is closely related to the tensorized Wasserstein distance through 
Benamou-Brenier formula~\cite{BB00}, see~\eqref{eq:Kanto-grad} later on.}
We suppose as in~\cite{CGM15} that at each time $t \ge 0$, 
the phase speeds $\bV = \left(\bv_1, \bv_2\right)$ is selected by the following 
\emph{steepest descent condition}:
\be\label{eq:varprob}
\bV \in \underset{\widetilde \bV=(\wt \bv_1,\wt \bv_2) \in \Zz(\bc)}{\text{argmin}}\left(\Ddd(\bc, \wt \bV) + \sup_{\bw \in \p\Ee(\bc)} 
 \sum_{i\in\{1,2\}} \int_\O c_i \wt \bv_i \cdot \grad w_i \d\x \right),
\ee
 {where $\Ee$'s subdifferential
\begin{align*}
  \p\Ee(\bc)= \left\{  {\bw\in L^2(\O;\R^d)^2} \middle| \Ee(\wh \bc) - \Ee(\bc) - 
  \int_\O \bw \cdot (\wh \bc - \bc) \d\x  \ge o\left(\|\wh \bc-\bc\|_{L^2(\O)}\right) \right\}
\end{align*}
 {is non-empty at each $\bc\in\bXx$ with $c_1\in H^2(\O)$ and $f(c_1)\in L^2(\O)$,
  and there, its elements $\bw=(w_1,w_2)$ are characterized by (c.f.~\cite{CGM15})}
\be\label{eq:dw}
w_1  - w_2 = 
- \alpha \Delta c_1 + \chi (1-2c_1)+ f(c_1) + (\Psi_1 - \Psi_2).
\ee
In formula~\eqref{eq:dw}, we have set $\alpha = \alpha_1 + \alpha_2$ and 
\be\label{eq:f}
f(c_1) = \log\left( \frac{c_1^{\theta_1}}{(1-c_1)^{\theta_2}}\right) = \theta_1 \log(c_1) - \theta_2 \log(c_2).
\ee
}

 {The integral in~\eqref{eq:varprob} has to be understood as the evaluation of the element $\bw=(w_1,w_2)$ in $\Ee$'s subdifferential at $\bc$
  on the infinitesimal changes $\p_t c_i=-\div(c_i \bv_i)$ induced by $\bV$.
  Since condition \eqref{eq:dw} only determines the difference $w_1-w_2$, 
  but not the $w_i$ individually, is easily seen that the maximum above is $+\infty$
  unless $\nabla\cdot(c_1\bv_1+c_2\bv_2)=0$ holds.
  That implies that any minimizer $\bV$ is such that the constraint \eqref{eq:c1+c2} is preserved.
}

 {In order to identify $\bV$, we swap minimization and maximization in the variational problem \eqref{eq:varprob}.
  The inner minimization is a quadratic problem in $\bV$, with solution}
\be\label{eq:v_i}
\bv_i = - {m_i} \grad w_i \quad \text{on}\; \{c_i >0\}\; \text{for}\; i \in \{1,2\}.
\ee
 {
  The outer maximization in then a quadratic problem for $\bw\in\p\Ee(\bc)$,
  which amounts to
}
\begin{align}
  \label{eq:elliptic}
  - \div\left( \sum_{i\in\{1,2\}} {m_i c_i} \grad w_i \right) = 0.  
\end{align}
 {
  Together with the condition \eqref{eq:dw}, this elliptic equation determines $\bw$ uniquely.
}

To sum up, the system of partial differential equations  {implied on the solution to
  the variational problem \eqref{eq:varprob}} is 
\begin{subequations}
  \label{eq:syst-PDE}
  \begin{align}
    \label{eq:syst1}
    \ds \p_t c_i - \div\left( {m_i c_i} \grad w_i\right) &= 0 \quad \text{for}\; i \in \{1,2\},\\
    \label{eq:syst2}
    c_1 + c_2 &= 1,  \\
    \label{eq:syst3}
    w_1 - w_2 &= -\alpha \Delta c_1 + \chi (1-2c_1)+ f(c_1) + (\Psi_1 - \Psi_2), 
    \end{align}
\end{subequations}
to be satisfied in $\O \times (0,\infty)$.
 {The system relation~\eqref{eq:syst-PDE} is complemented 
  by homogeneous Neumann boundary conditions}
\be\label{eq:dcdn0}
-\grad c_1 \cdot \n = 0 \quad \text{on}\; \p\O \times \R_+
\ee
 {and the no-flux conditions \eqref{eq:imper}. 
  Notice that with these boundary conditions,
  the set \eqref{eq:syst-PDE} of equations is (formally) closed in the sense that 
  if $\bc$ is known at some instance of time,
  then the $w_i$ are uniquely determined (up to irrelevant global additive constants)
  by \eqref{eq:syst3} and by the elliptic equation \eqref{eq:elliptic} 
  that follows from adding \eqref{eq:syst1} for $i=1,2$ 
  in combination with the conservation $\p_t(c_1+c_2)=0$ implied by \eqref{eq:syst2}.
  Hence $\bc$'s time derivatives $\partial_tc_i$ are determined as well.}

 {At this point, it is natural to}
define the chemical potential $\mu_i$ of the phase $i$ by 
\be\label{eq:mu_i}
\mu_i = w_i - \theta_i \log(c_i) - \Psi_i,  \qquad i\in\{1,2\},
\ee
so that \eqref{eq:syst3} turns to 
\be\label{eq:dmu}
\mu_1  - \mu_2 = 
- \alpha \Delta c_1 + \chi (1-2c_1).
\ee
 {So far}, $\bmu {=(\mu_1,\mu_2)}$ is only defined up to an additive constant. This degree of freedom is eliminated by 
imposing for almost all $t>0$ that 
\be\label{eq:ov-mu}
\int_\O \ov \mu(\x,t) \d\x  = 0, \quad \text{where}\; \ov \mu = c_1 \mu_1 + c_2 \mu_2.
\ee
 {Whilst the phase chemical potential $\mu_i$ have a clear physical sense only on $\{c_i>0\}$, 
the global chemical potential $\ov \mu$ remains meaningful in the whole $\O$. In particular, its spatial variations can be controlled, and thus $\ov \mu$ itself
too thanks to a Poincar\'e-Wirtinger inequality.}

 {With the help of the chemical potentials, the system \eqref{eq:syst-PDE} turns into~\eqref{eq:syst-PDE2}, 
to be complemented with boundary conditions~\eqref{eq:BC} and initial conditions~\eqref{eq:initial}.}

In the next section, we highlight some differences between the non-local model~\eqref{eq:syst-PDE2} 
and the local degenerate Cahn-Hilliard model that has been studied for instance in~\cite{EG96}.

\subsection{Comparison with the classical degenerate Cahn-Hilliard model}\label{ssec:comparison}

Even in the simple situation where the external potentials $\Psi_i$ are equal to $0$ and where  $\theta_1 = \theta_2 =0$, 
the system~\eqref{eq:syst-PDE} differs as soon as $d\ge2$ from the local degenerate Cahn-Hilliard model that can be written as 
\be\label{eq:EG}
\p_t c - \div(\eta(c) \grad \mu) = 0, \qquad \mu = - \alpha \Delta c + \chi(1-2c), \qquad \eta(c) = \frac{m_1 m_2 c(1-c)}{m_1 c + m_2 (1-c)}.
\ee
We refer to \cite{EG96} for the existence of weak solutions to~\eqref{eq:EG} (complemented with suitable boundary conditions)
and to \cite{BBG01} for the extension of the model to 
the case of N phases ($N\ge3$). Here, $\mu$ is the generalized chemical potential that is defined as the difference of the phase chemical potentials.

The energy $\widetilde \Ee$ associated to~\eqref{eq:EG} is similar to the one of our problem, i.e., 
$$
\widetilde \Ee(c) = \frac{\alpha}2 \int_\O |\grad c|^2 \d\x + \chi \int_\O c(1-c) \d\x.
$$
But both the equation governing the motion~\eqref{eq:cons-i} and the dissipation~\eqref{eq:Ddd} have to be modified. 
More precisely, the continuity equation~\eqref{eq:cons-i} must be replaced by its nonlinear counterpart 
$$
\p_t c  + \div(\eta(c) \bv) = 0, 
$$
while the dissipation is now given by 
$$
\widetilde \Ddd(c,\bv) = \int_\O \eta(c) |\bv|^2 \d\x.
$$
Therefore, the PDEs~\eqref{eq:EG} can still be interpreted as the gradient flow of the energy $\widetilde \Ee$, but 
the geometry is different: rather than considering some classical quadratic Wasserstein distance for each phase and to constrain the 
sum of the concentrations to be equal to 1 (as it will be the case for our approach), the set
 $$\Aa = \left\{ c \in L^1(\O; \R_+) \; \middle|  \; \int_\O c(\x,t)\d\x = \int_\O c^0_1(\x) \d\x \right\}$$ 
has to be equipped with the weighted Wasserstein metric corresponding to the concave mobility $\eta$. 
We refer to~\cite{DNS09} for the description of the corresponding metric and to~\cite{LMS12} for the rigorous 
recovery of~\eqref{eq:EG} by a gradient flow approach.

The difference between the non-local model \eqref{eq:syst-PDE} and the local one~\eqref{eq:EG} can also be seen as follows. 
Summing the first equation of~\eqref{eq:syst-PDE} for $i\in\{1,2\}$ yields 
\be\label{eq:v_tot}
\div\bJ_{\rm tot} = 0, \quad \text{where}\quad 
\bJ_{\rm tot} = \sum_{i\in\{1,2\}} c_i \bv_i =  -\sum_{i\in\{1,2\}} m_i c_i \grad \mu_i. 
\ee
The equation for $c_1$ can then be rewritten 
\be\label{eq:v_tot_c1}
\p_t c_1 + \div\left( \rho(c_1) \bJ_{\rm tot} - \eta(c_1) \grad (\mu_{1}- \mu_2)\right)=0, 
\quad \text{with} \quad 
\rho(c) = \frac{m_1 c}{m_1 c + m_2 (1-c)}.
\ee
Thus our model~\eqref{eq:syst-PDE} boils down to the local Cahn-Hilliard equation as soon as $\bJ_{\rm tot} \equiv \0$. 
This is the case when $d=1$ because of~\eqref{eq:v_tot}, but no longer if $d\ge 2$. 
Since our non-local model does not impose that  $\bJ_1 = -\bJ_2$, it allows for additional motions.
These motions ---corresponding to the transport term $ \div\left( \rho(c_1) \bJ_{\rm tot}\right)$ in~\eqref{eq:v_tot_c1}--- contribute to the dissipation 
as shows the formula 
$$
\Ddd(\bc, -\grad \bmu) = \int_\O \frac{|\bJ_{\rm tot}|^2}{m_1 c_1 + m_2 c_2}\d\x + \widetilde \Ddd(c_1, -\grad (\mu_1 - \mu_2)).
$$
Therefore, and as already noticed by Otto and E in~\cite{OE97}, the instantaneous dissipation corresponding to 
a phase configuration $\bc$ is greater for the non-local model~\eqref{eq:syst-PDE} than for the local model~\eqref{eq:EG}
and the energy decreases faster.

\subsection{Numerical illustration}\label{ssec:num}
The goal of this section is to illustrate the behavior of the model~\eqref{eq:syst-PDE} and to compare it 
with the classical degenerate Cahn-Hilliard problem~\eqref{eq:EG}.
In order to solve numerically \eqref{eq:syst-PDE} we use an implicit in time finite volume scheme 
with upstream mobility described in~\cite{CN_FVCA} and inspired from the oil engineering context~\cite{EHM03}. 
The mesh is triangular and assumed to fulfill the so-called orthogonality condition~\cite{Herbin95,EGH00} 
(this amounts to requiring the mesh to be Delaunay) so that 
the diffusive fluxes can be approximated thanks to a two-point flux approximation in a consistent way. 
As it is exposed in~\cite{CN_FVCA}, the scheme 
is positivity preserving (i.e., $0 \leq c_{i,h} \leq 1$), it is energy diminishing (the discrete counterpart of the energy is decreasing)
and entropy stable. It leads to a nonlinear system of algebraic equations to be solved at each time step. 
It is shown in~\cite{CN_FVCA} that this system admits (at least) one solution that is computed thanks to 
the Newton-Raphson method.

Concerning the problem~\eqref{eq:EG}, we use a similar approach, but since the mobility function $\eta$ is 
no longer monotone, we have to use an implicit Godunov scheme to discretize it as a generalization of the upstream mobility 
(see for instance~\cite{CG16_MCOM}). Here again, the discrete solution remains bounded between $0$ and $1$, 
the energy is decreasing and the entropy remains bounded. Here again, the resulting nonlinear system 
is solved at each time step by the mean of the Newton-Raphson method.

\begin{rem}
Alternative numerical methods have been proposed in order to solve degenerate Cahn-Hilliard 
problems. We won't perform here an exhaustive list, but let us mention the contributions of Barrett {\em et al.} 
based on conformal finite elements \cite{BBG99, BBG01}. Even though very efficient, these methods 
have the drawback of requiring a small stabilization to be tuned following the mesh size. 
Unless one considers non-smooth energies as in \cite{BE92}, the scheme does not preserve the bounds 
$0 \leq c_{i,h} \leq 1$. These difficulties are overpassed in our approach 
by using some entropy stable hyperbolic fluxes to discretize the mobilities.  

Since our model has a Wasserstein gradient flow structure, it would be natural 
to use a Lagrangian method as for instance in~\cite{BCC08, MO17, JMO17, CDMM17}. The main problem with this approach 
is that both phase move with their own speed, therefore such an approach would impose to 
move two meshes simultaneously. It is then rather unclear how to manage the constraint~\eqref{eq:c1+c2} 
in this case. For this reason, it seems more suitable to stick to an Eulerian description. 
An alternative approach to solve numerically our problem would therefore be to 
adapt the ALG2-JKO algorithm of Benamou {\em et al.}~\cite{BCL16} to our setting {, see~\cite{CGLM_preprint}.}
\end{rem}

We propose two different test cases that will allow to illustrate the difference between the local model~\eqref{eq:EG} 
and the non-local model~\eqref{eq:syst-PDE}. For both of them, we do not consider any exterior potential, i.e., $\Psi_i = 0$, 
and we neglect the thermal diffusion, i.e., $\theta_i = 0$. Both phase mobilities $m_i$ are assumed to be equal to $1$.

\subsubsection{Test case 1: from a cross to a circle} 
We start from an initial data that is the characteristic function of a cross and we choose $\alpha = (3.6).10^{-4}$ and $\chi = 0.8$. 
Since $\alpha \ll \chi$, it follows from the Modica and Mortola' result ~\cite{MM80} that the free energy is close to the perimeter 
of a characteristic set (up to a multiplicative constant). This means that (up to a small regularization) both the local 
and the non-local Cahn-Hilliard models aim at minimizing the perimeter of the sets $\{c_1 = 0\}$ and $\{c_1 = 1\}$ 
corresponding to pure phases. Since the non-local model allows for more movements (cf. Section~\ref{ssec:comparison}), 
the energy (thus the perimeter) should decay faster for the nonlocal model. This is indeed what we observe on Figures~\ref{fig:cross_NRJ} and~\ref{fig:cross_snapshots}.

\begin{figure}[htb]
\centering
\includegraphics[width=7cm]{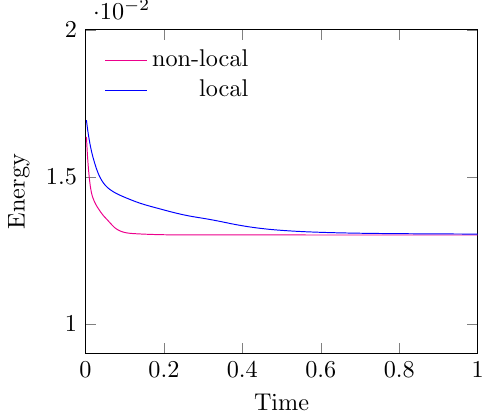}
  \caption{Evolution of the energy $\Ee(\bc)$ along time for the non-local model~\eqref{eq:syst-PDE} and the local one~\eqref{eq:EG}. 
  The decay of the energy is faster for the non-local model, as shown in Section~\ref{ssec:comparison}.}
  \label{fig:cross_NRJ}
 \end{figure}

\begin{figure}[htb]
\centering
\includegraphics[width = 4cm]{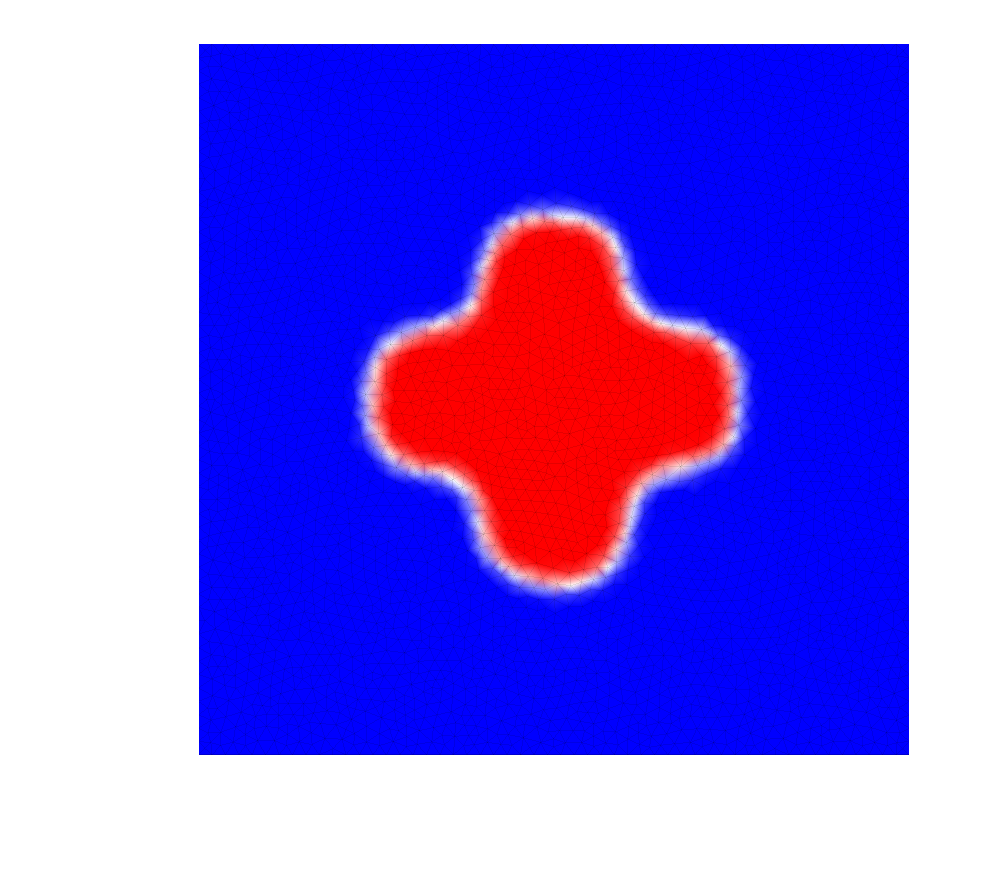}
\includegraphics[width = 4cm]{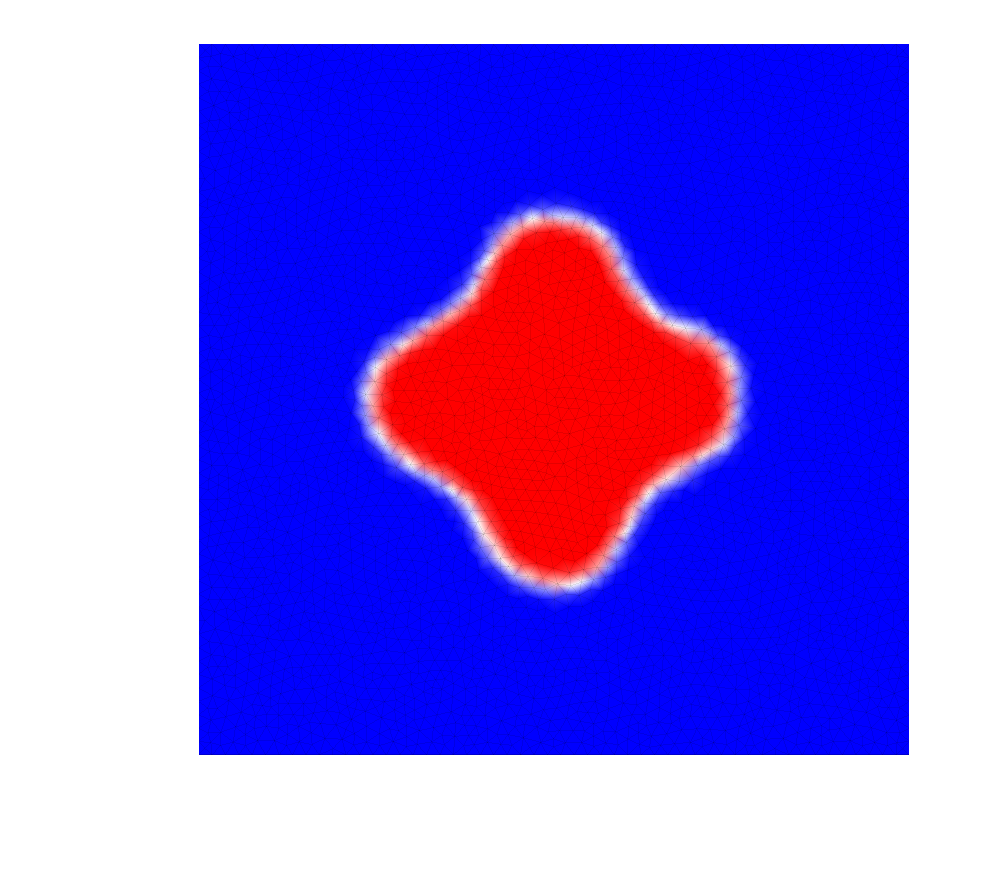}
\includegraphics[width = 4cm]{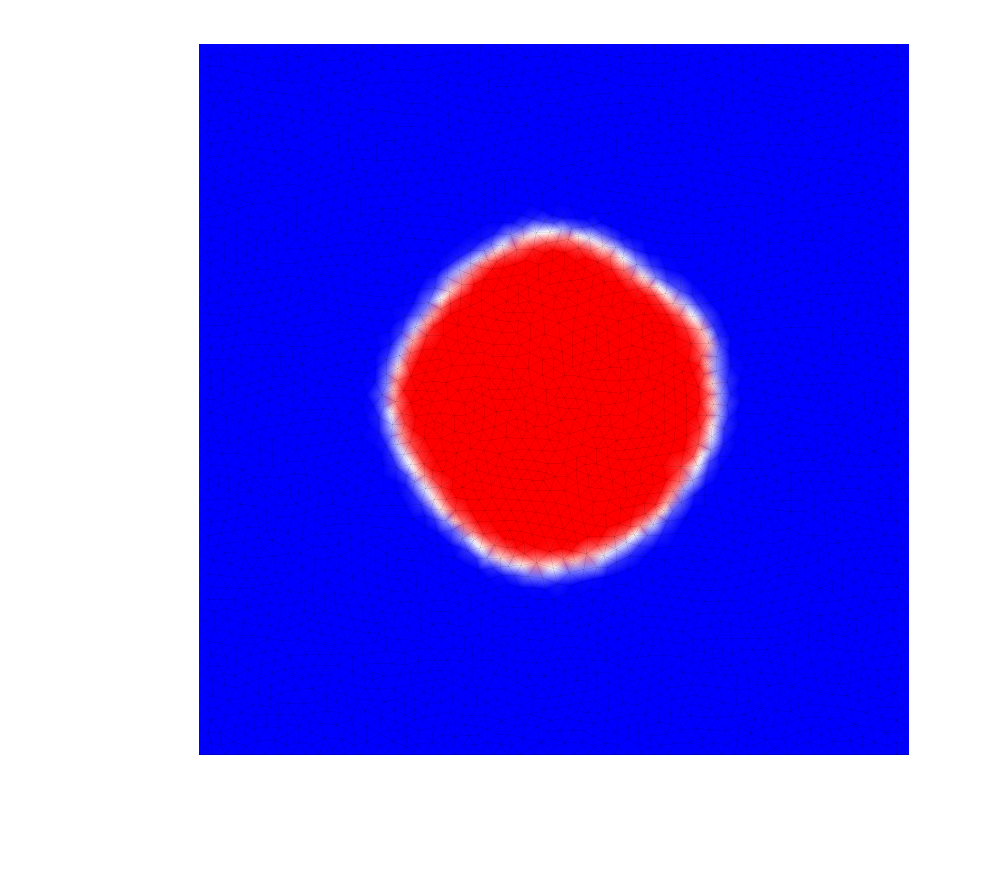}
\\
\includegraphics[width = 4cm]{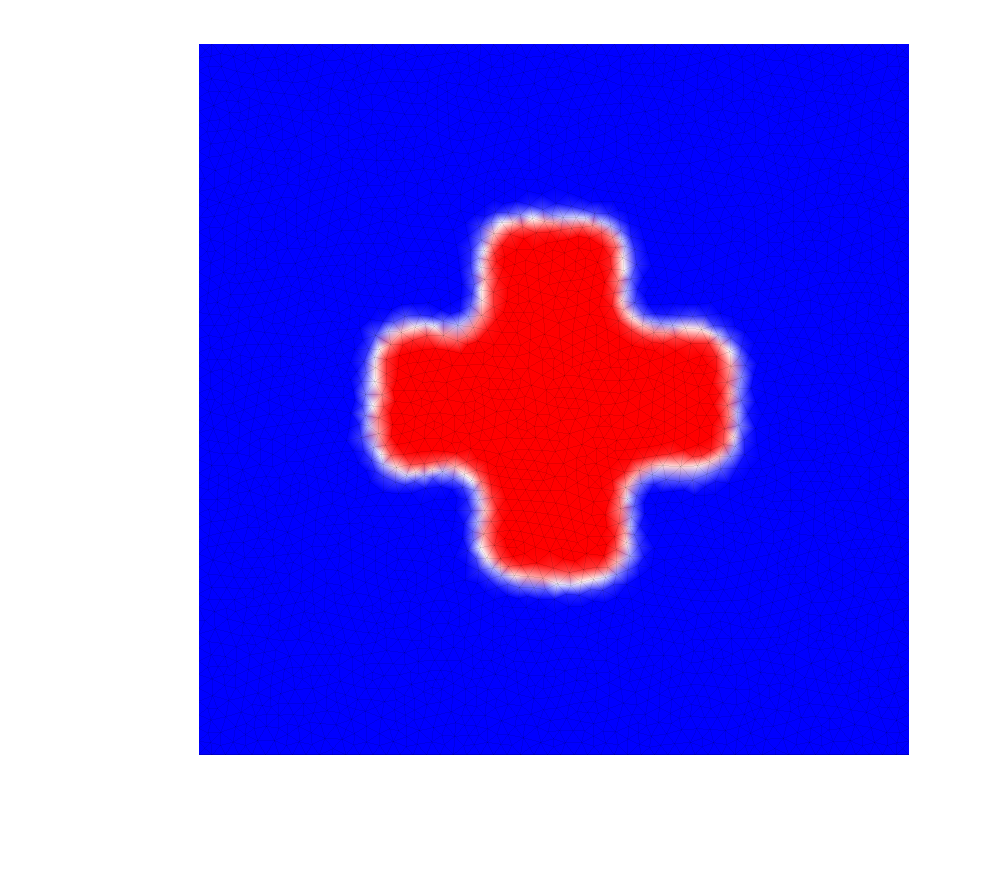}
\includegraphics[width = 4cm]{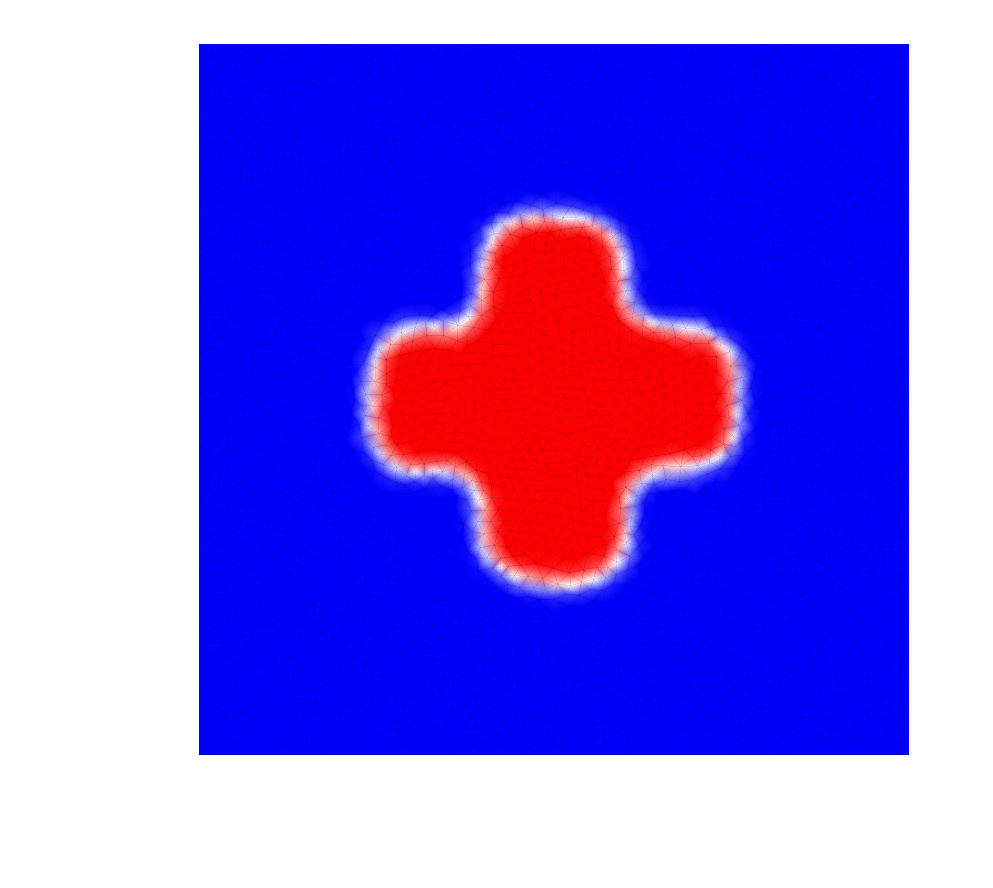}
\includegraphics[width = 4cm]{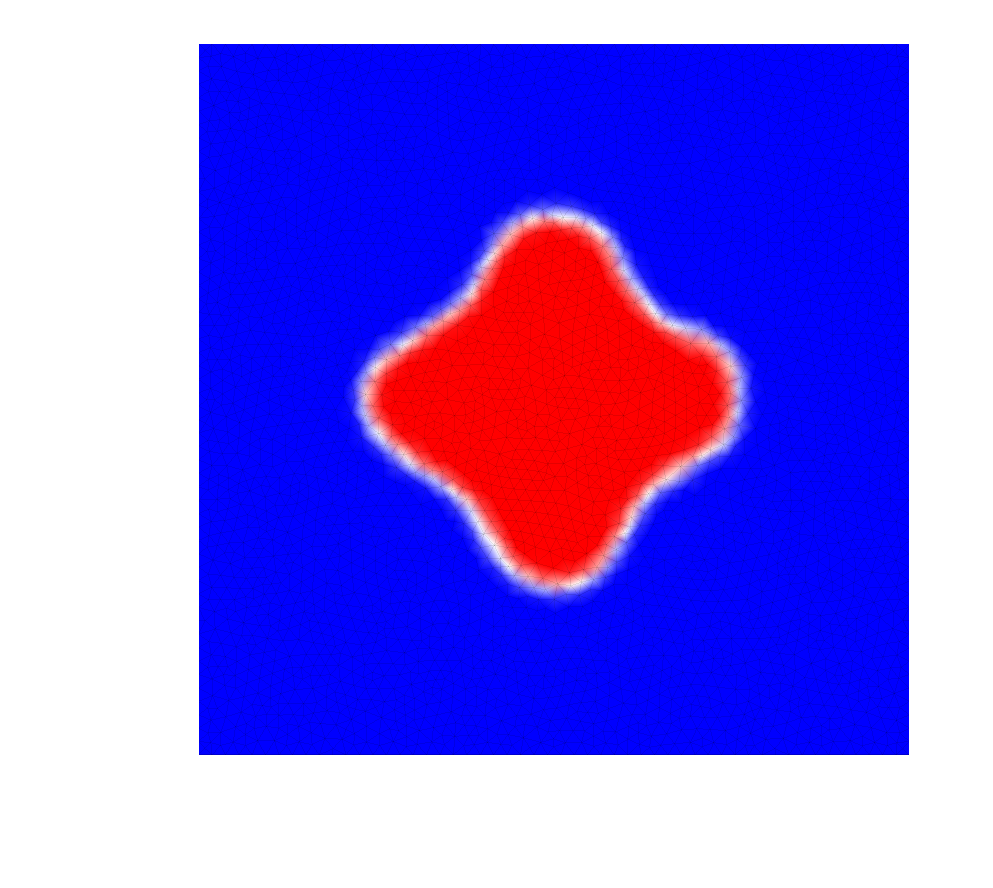}
\caption{Evolution along time of the numerical solution to the non-local problem~\eqref{eq:syst-PDE} 
(top) and to the local problem~\eqref{eq:EG} (bottom).
Snapshots at time $t = 10^{-2}$ (left), $t=2.10^{-2}$ (middle), and $t=10^{-1}$ (right).} \label{fig:cross_snapshots}
\end{figure}

\subsubsection{Test case 2: Spinodal decomposition}
Similarly to the local model, the non-local model is able to reproduce the spinodal decomposition for mixtures. 
In order to illustrate this fact, we start from an initial data which consists in a constant concentration plus a small 
random perturbation: 
\[
c_1^0(\x) = 0.5 + r(\x), \qquad r \ll 1. 
\]
Since $c_1 = 0.5$ is very unfavorable from an energetic point of view, both phase will separate very rapidly, letting areas 
with pure phase appear. Then these area will cluster in order to minimize their perimeter. We choose $\alpha = 3.10^{-4}$ and $\chi = 0.96$ and 
we plot on Figure~\ref{fig:spinodal_snapshots}
some snapshots to illustrate the spinodal decomposition corresponding to models~\eqref{eq:syst-PDE} and \eqref{eq:EG}. 
On Figure~\ref{fig:spinodal_NRJ}, we compare the evolution of the energy along time for spinodal decomposition 
corresponding to both models. As expected, the energy decay is faster for the non-local model than for the local one. 
But contrarily to Test case 1, the solutions seem to converge towards different steady states. 

\begin{figure}[htb]
\centering
\includegraphics[width = 4cm]{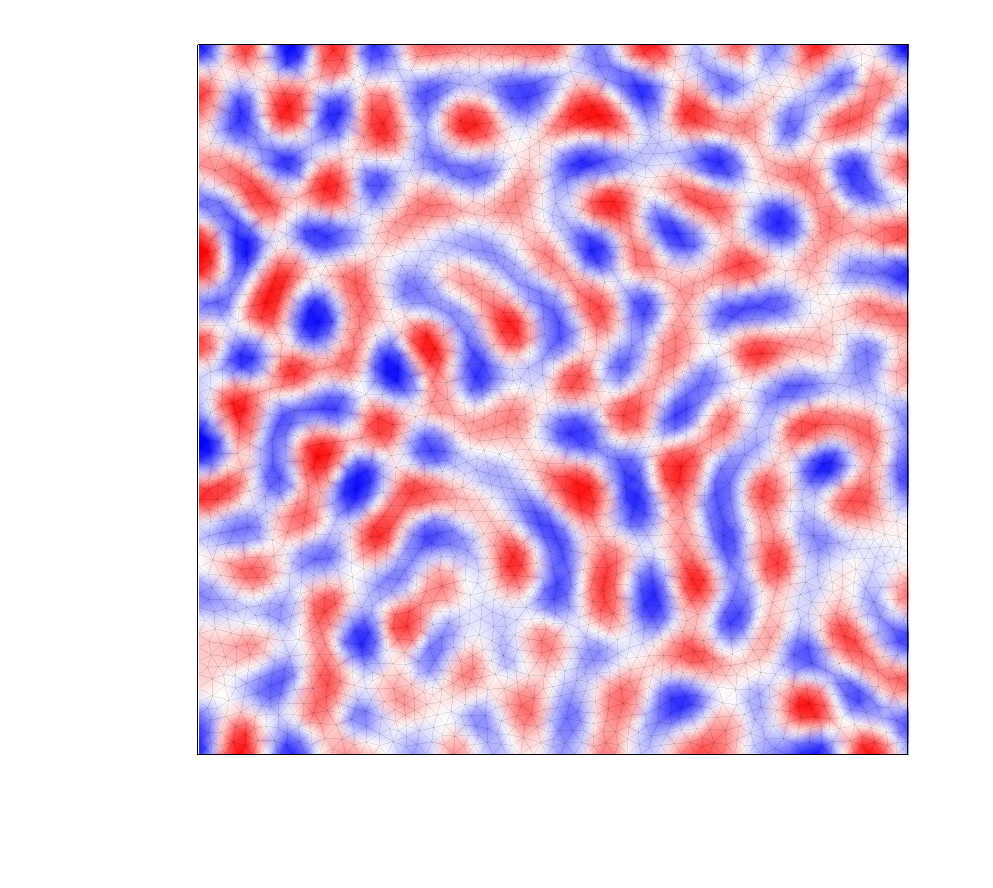}
\includegraphics[width = 4cm]{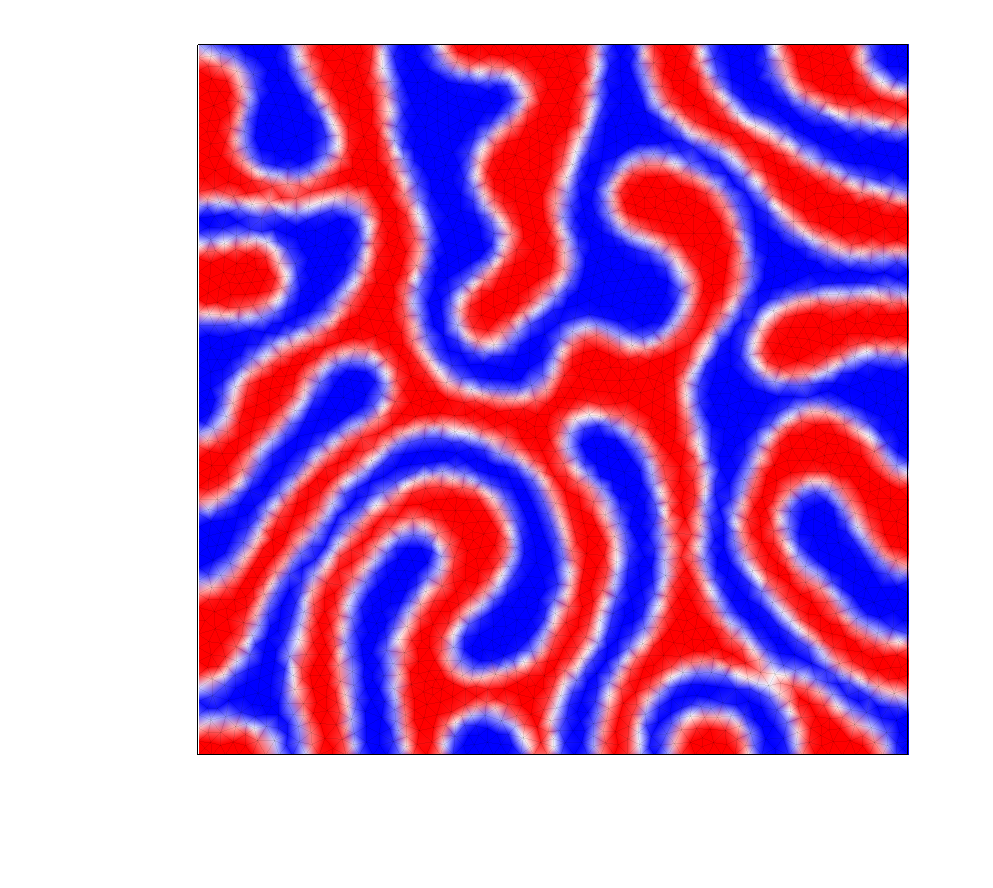}
\includegraphics[width = 4cm]{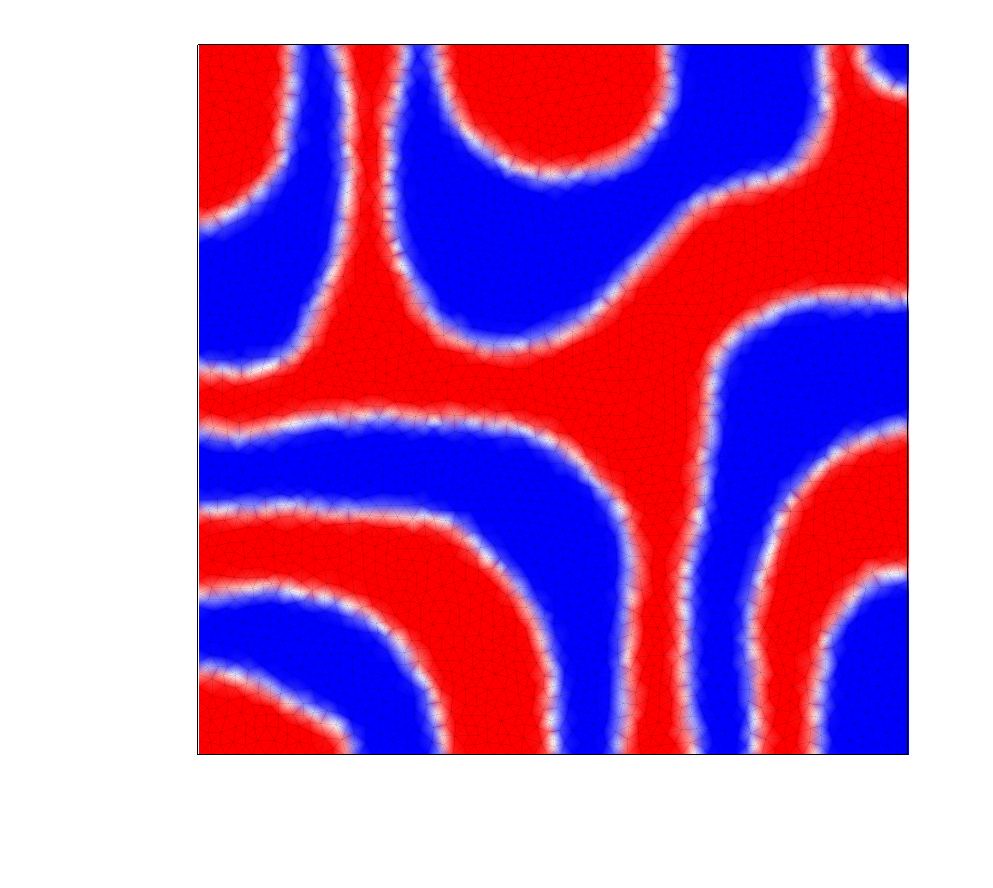}
\\
\includegraphics[width = 4cm]{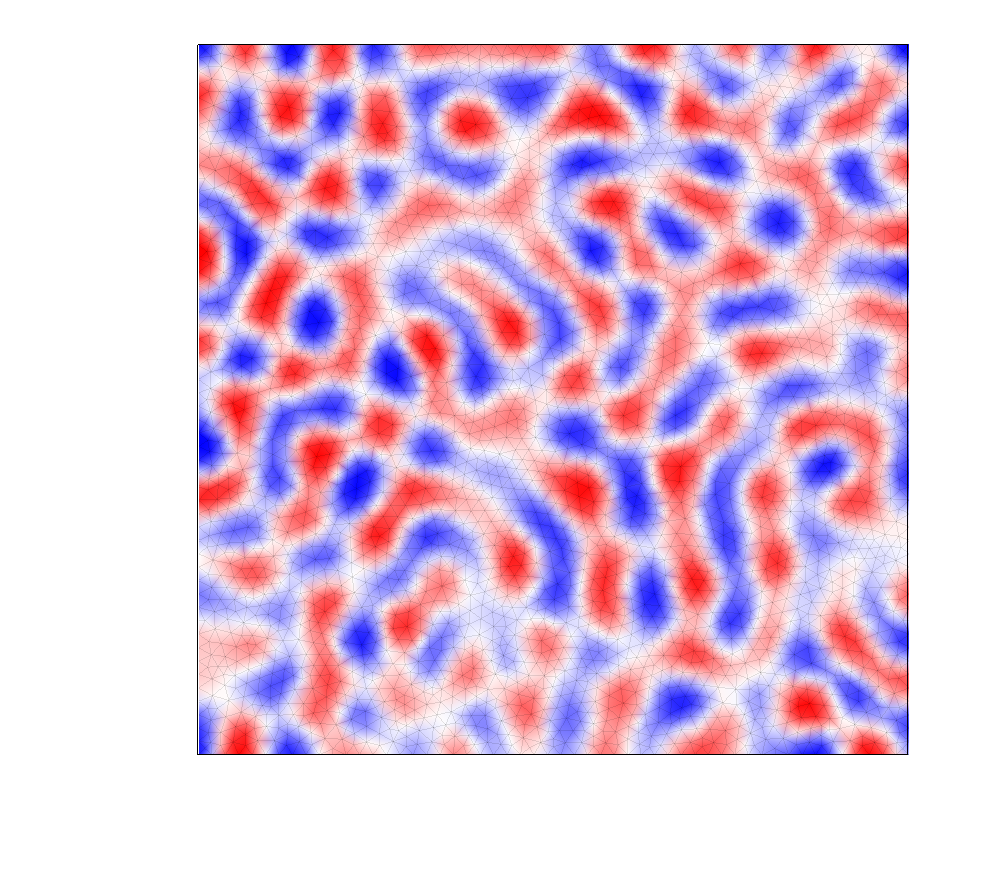}
\includegraphics[width = 4cm]{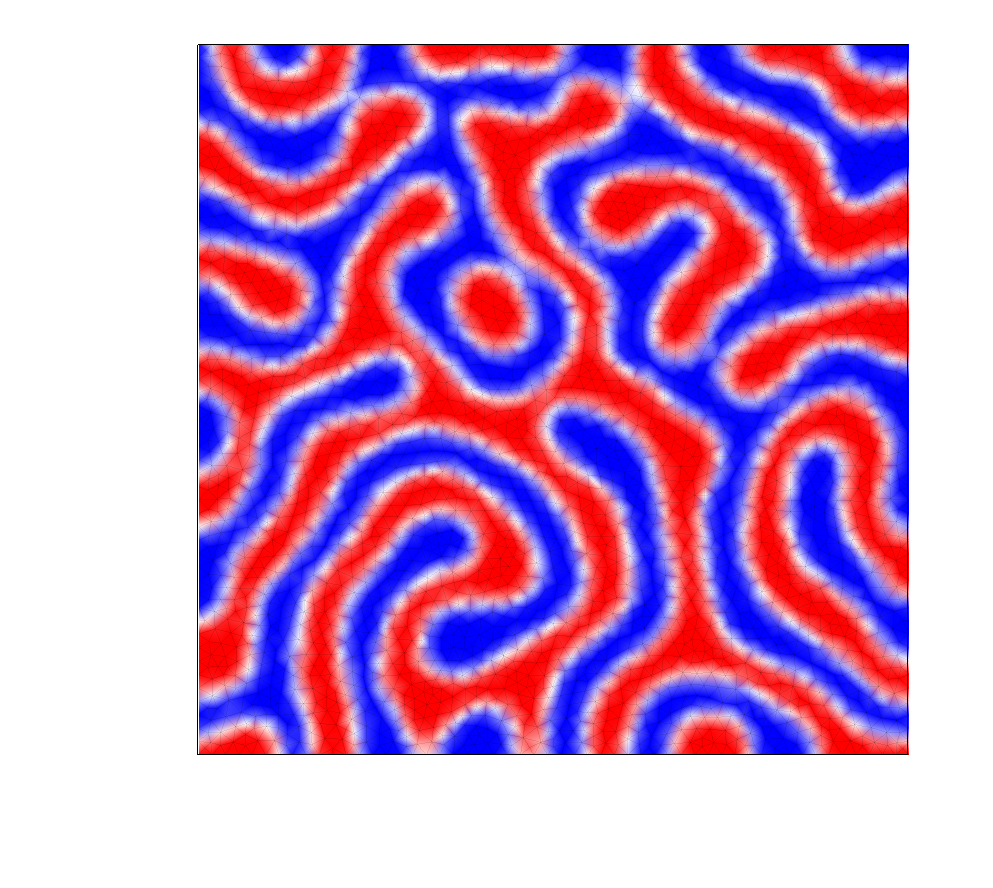}
\includegraphics[width = 4cm]{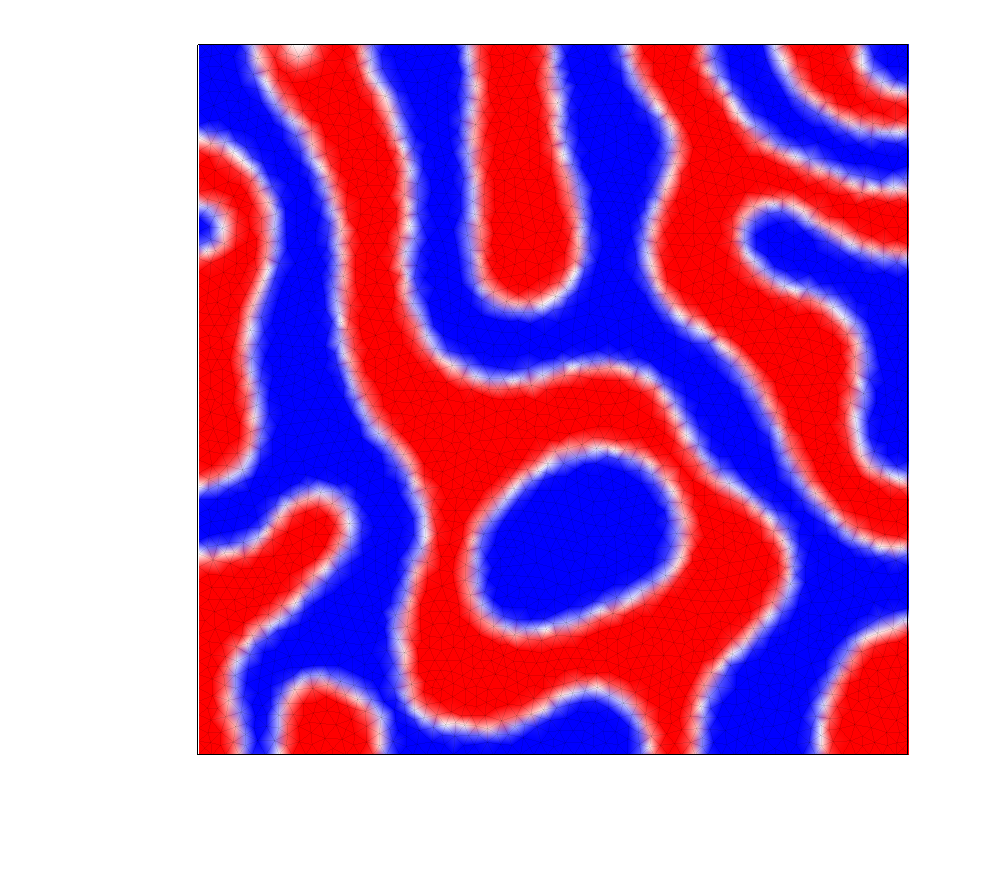}
\caption{Snapshots at times $t=6.10^{-3}$ (left), $t=5.10^{-2}$ (middle), and 
$t =1$ (right) illustrating the spinodal decomposition governed by model~\eqref{eq:syst-PDE} (top) 
and model~\eqref{eq:EG} (bottom). }\label{fig:spinodal_snapshots}
\end{figure}

\begin{figure}[htb]
\centering
\includegraphics[width=7cm]{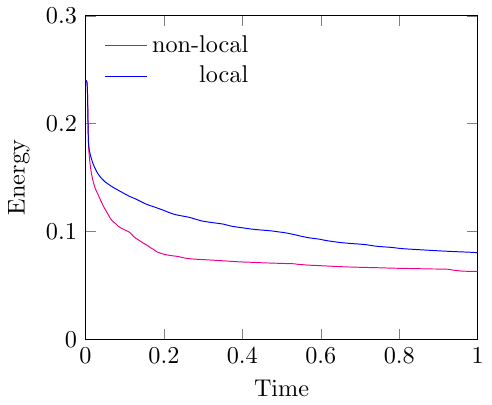}
 \caption{Evolution of the energy $\Ee(\bc)$ along time for the non-local model~\eqref{eq:syst-PDE} and the local one~\eqref{eq:EG} 
 for the spinodal decomposition test case.}\label{fig:spinodal_NRJ}
 \end{figure}

\section{Wasserstein gradient flow, JKO scheme and main result}\label{sec:JKO}

\subsection{Wasserstein distance}\label{ssec:Wasserstein}

As a preliminary to the introduction of the minimizing movement scheme, 
we introduce some necessary material related to Wasserstein (or Monge-Kantorovich) distances
between nonnegative measures of prescribed mass that are absolutely continuous w.r.t. to the Lebesgue measure. 
We refer to Santambrogio's monograph~\cite{Santambrogio_OTAM} for an introduction to optimal transportation and to the Wasserstein distances, 
and to Villani's big book~\cite{Villani09} for a more complete presentation. 

Given  two elements $c_i$ and $\check c_i$  of $\Aa_i$ ($i\in\{1,2\}$), 
a map $\bt: \O \to \O$ is said to send $c_i$ on $\check c_i$ (we write $\check c_i = \bt \# c_i$) if
$$
\int_A \check c_i(\x) \d\x = \int_{\bt^{-1}(A)} c_i(\x) \d\x, \qquad \text{for all Borel subset $A$ of $\O$}.
$$
The Wasserstein distance $W_i(c_i, \check c_i)$ with quadratic cost function 
between $c_i$ and $\check c_i$ is then defined by  
\be\label{eq:W_i}
W_{i}(c_i, \check c_i) = \inf_{\bt \;\text{s.t.}\; \check c_i = \bt\#c_i}\left( \frac1{m_i}\int_\O |\x-\bt(\x)|^2 c_i(\x)\d\x \right)^{1/2}, \qquad i \in \{1,2\}.
\ee
In~\eqref{eq:W_i}, the infimum is in fact a minimum, and $\bt$ is the gradient of a convex function.
In our context of fluid flows, the cost for moving the mass of the phase $i$ from a configuration $c_i$ to another configuration $\check c_i$ 
regardless to the other phase is equal to $W_i^2(c_i, \check c_i)$. The multiplying factor $1/m_i$ is natural since
the more mobile is the phase, the less expensive are its displacements.  
We can then define the tensorized Wasserstein distance $\bW$ on $\bAa$ by
$$
\bW^2(\bc, \check \bc) = W_1^2(c_1, \check c_1) + W_2^2(c_2, \check c_2), \qquad \forall \bc, \check \bc \in \bAa.
$$
In the core of the proof, we will make an extensive use of the Kantorovich dual problem.
More precisely, we will use the fact that 
\be\label{eq:DP}
\frac12 W_i^2(c_i, \check c_i) = \sup_{\substack{
\varphi_i \in L^1(c_i), \psi_i \in L^1(\check c_i) 
\\ \varphi_i(\x) + \psi_i(\by) \leq \frac{|\x - \by|^2}{2m_i}
}} \int_\O c_i(\x) \varphi_i(\x) \d\x + \int_\O \check c_i(\by) \psi_i (\by) \d \by. 
\ee
Here, $L^1(\rho)$ denotes the sets of integrable functions for the measure with density $\rho$. 
Here again, the supremum is in fact a maximum, and the 
Kantorovich potentials $(\varphi_i, \psi_i)$ achieving the $\sup$ in~\eqref{eq:DP} are 
$\d c_i \otimes \d \check c_i$ unique up to an additive constant. 
The optimal transportation $\bt_i$ sending $c_i$ on 
$\check c_i$ achieving the $\inf$ in~\eqref{eq:W_i} is related to the Kantorovich potential by 
$$
\bt_i(\x) = \x - m_i \grad \varphi_i(\x), \qquad \forall \x \in \O
$$
with $\varphi_i$ achieving the sup in~\eqref{eq:DP}. As a consequence, the formula~\eqref{eq:W_i} provides 
\be\label{eq:Kanto-grad}
W_i^2(c_i, \check c_i) = m_i \int_\O c_i \left|\grad \varphi_i \right|^2 \d\x, \qquad i \in \{1,2\}
\ee
to be used in the sequel.

\subsection{The JKO scheme and the approximate solution}\label{ssec:JKO}

We have now all the necessary material at hand to define the minimizing movement scheme. 
Let $\tau>0$ and $\bc^{n-1} \in \bAa\cap \bXx$, then define the functional 
$\Ff^n_\tau: \bAa \to \R\cup \{+\infty\}$  by setting
\[
  \Ff^n_\tau(\bc)  = \Ee(\bc) + \frac{1}{2\tau} \bW^2(\bc, \bc^{n-1}), \qquad \forall \bc \in \bAa.
\]
The functional $\Ff^n_\tau$ is bounded from below since all its components are.
Then we define 
\be\label{eq:JKO}
\bc^n \in \underset{\bc \in \bAa}{\text{argmin}}\;\Ff^n_\tau(\bc).
\ee
 {In case that $\Ff_\tau^n$ possesses several mimimizers, a selection must be made:
  if $c^{n-1}_1$ and $c^{n-1}_2$ are both positive a.e. in $\O$, then any minimizer is fine.
  However, if one of the two densities vanishes on a set of positive measure,
  then we need to select $\bc^n$ such that it is an accumulation point in the weak $H^2(\O)^2$-topology for $\delta\to0$ 
  of the set of minimizers of the functionals}
\begin{align}
  \label{eq:Ffdelta}
  \Ff^{n,\delta}_\tau(\bc) = \Ee(\bc) + \frac{1}{2\tau} \bW^2(\bc, \bc^{n-1,\delta}),
  \quad \text{with}\quad
  c^{n-1,\delta}_i=(1-\delta)c^{n-1}_i + \delta \ov c_i
\end{align}
where $\ov c_i$ is defined by~\eqref{eq:m_i}.
Under certain conditions, 
for instance in the thermally agitated situation where $\theta_1,\theta_2>0$,
it is known a priori that both $c^n_1$ and $c^n_2$ are strictly positive,
and hence no selection is necessary,
see e.g. Lemma \ref{lem:log(c)}.
 {The proof of  {the existence of $\bc^n$ solution to} \eqref{eq:JKO} is given below.
  The proof that one can indeed select the minimizer in the way described above
  is postponed to Corollary \ref{cor:JKO} in Section \ref{sec:Euler},
  where we show compactness properties of the set of minimizers.} 
\begin{prop}\label{prop:JKO}
  For any $\bc^{n-1} \in \bAa\cap \bXx$, there exists (at least) one solution 
  $\bc^n \in \bAa\cap \bXx$ to the minimization scheme~\eqref{eq:JKO}. 
\end{prop}
\begin{proof}
  Let $\left(\bc^{n,k}\right)_{k\ge0}$ be a minimizing sequence with $\bc^{n,0} = \bc^{n-1}$
   {and such that}
  \be\label{eq:mini-Eee}
  \Ee(\bc^{n,k}) \le \Ee(\bc^{n-1}) < \infty \qquad \text{for all $k \ge 0$}. 
  \ee
  We infer from~\eqref{eq:bXx} that $\bc^{n,k} \in \bXx$, 
  hence $0 \le c_i^{n,k} \le 1$ for all $k \ge 0$ and $i \in \{1,2\}$.
  Moreover,  {in view of the presence of the Dirichlet energy in $\Ee$,
    it follows by means of the Poincar\'e-Wirtinger inequality 
    that the $c_i^{n,k}$ are $k$-uniformly bounded in $H^1(\O)$.
    Passing to a subsequence if necessary, the $c_i^{n,k}$ converge to limits $c_i^n$,
    weakly in $H^1(\O)$ thanks to Alaoglu's theorem, 
    and also strongly in $L^1(\O)$ by Rellich's compactness lemma.
    All the components of $\Ee$ are sequentielly lower semi-continuous with respect to this convergence,
    and so}
  $$
  \Ee(\bc^n) \le \operatorname*{lim\,inf}_{k\to\infty}\Ee(\bc^{n,k}).
  $$
   {Finally, also the Wasserstein distance is lower semi-continuous 
    with respect to strong convergence in $L^1(\O)$,}
  $$
  \bW(\bc^{n}, \bc^{n-1}) \le  \operatorname*{lim\,inf}_{k\to\infty}\bW(\bc^{n,k}, \bc^{n-1}).
  $$
  As a consequence, $\bc^n$ is a minimizer of $\Ff_\tau^n$.
\end{proof}
 {Each sequence $\left(\bc^n\right)_{n\ge 1}$ of iterated solutions to the scheme~\eqref{eq:JKO}
  is accompanied by approximate phase potentials $\left(\bmu^n\right)_{n\ge1}$:
  these are introduced in such a way that in the time-continuous limit, 
  the continuity equations $\p_tc_i=\nabla\cdot[m_ic_i\nabla(\mu_i+\Psi_i)]$ hold.
  The suitable quantities $\mu_i^n$ are identified (somewhat \textit{a posteriori}) 
  by comparing the optimality conditions for \eqref{eq:JKO} with the persued limit PDE system \eqref{eq:syst-PDE2}.}

 {In order to justify the formal definition of $\bmu^n$ that we give below in \eqref{eq:defmu}, 
  we anticipate an auxiliary result from Section \ref{sec:Euler},
  that can be understood as a formulation of the time-discrete Euler-Lagrange equations.
  It involves the (backward) Kantorovich potentials
  $\bvarphi^n = \left(\varphi_1^n, \varphi_2^n\right)$ 
  sending $\bc^n$ on $\bc^{n-1}$.
  A subtle point is to overcome the inherent non-uniqueness of $\varphi_1^n$ and $\varphi_2^n$
  --- particularly if one of the $c^n_i$'s vanishes on a set of positive measure ---
  by making a suitable selection, as will be explained in Section \ref{sec:Euler}.
  We note that the intricate selection procedure for the minimizer $\bc^n$ in \eqref{eq:JKO} enters precisely at this point.}
\begin{lem}\label{lem:preEL}
  At each step $n\ge1$, 
  there exist Kantorovich potentials $\bvarphi^n=(\varphi_1^n,\varphi_2^n)$ for sending $\bc^n$ to $\bc^{n-1}$
  such that $F^n:\Omega\to\R$, given by
  \be\label{eq:F^n}
  F^n = \frac{\varphi_1^n}{\tau} -  \frac{\varphi_2^n}{\tau}  - \alpha \Delta c_1^n + f(c_1^n) + \chi(1-2c_1^n) + \Psi_1-\Psi_2,
  \ee
  satisfies
  \begin{align}
    \label{eq:EL}
    F^n\le 0 \quad \text{a.e. in $\{c_1^n>0\}$}, \qquad 
    F^n\ge 0 \quad \text{a.e. in $\{c_2^n>0\}$}.    
  \end{align}
\end{lem}
 {With the particular choice of $\bvarphi^n$ from the lemma, 
  we define $\bmu^n = \left(\mu_1^n, \mu_2^n\right) : \O \to \R^2$ by}
\begin{align}
  \label{eq:defmu}
  \mu_1^n := - \frac{\varphi_1^n}\tau - \Psi_1 - \theta_1 \log(c_1^n)  + \left(F^n\right)_+ , \quad \text{and} \quad
  \mu_2^n := - \frac{\varphi_2^n}\tau  - \Psi_2 - \theta_2 \log(c_2^n)  + \left(F^n\right)_- .
\end{align}
 {Actually, since $F^n$ is invariant under simultaneous addition of
  a global constant to both $\varphi_1^n$ and $\varphi_2^n$, 
  we can assume that their normalization is chosen such that}
\begin{align}
  \label{eq:ov-mu_mean}
  \int_\O \ov \mu^n \d\x = 0, &\qquad 
                                  \text{where}\quad \ov \mu^n = c_1^n \mu_1^n + c_2^n \mu_2^n.    
\end{align}
 {Note that it follows directly from the definition of $\bmu^n$ that}
\begin{align}
    \label{eq:diff-mu}
    \mu_1^n -\mu_2^n = - \alpha \Delta c_1^n + \chi(1-2c_1^n) &\qquad
                                                                \text{a.e. in $\O$}.  
\end{align}
 {Moreover, thanks to \eqref{eq:EL}, we have that}
\begin{align}
  \label{eq:mu_i^n}
  \mu_i^n = - \frac{\varphi_i^n}\tau - \Psi_i - \theta_i \log(c_i^n)  &\qquad 
                                                                        \text{a.e. in $\{c_i^n>0\}$}.
\end{align}
 {With these quantities at hand, we define further the  {piecewise constant} interpolants} $\bc_\tau: \R_+ \to \bAa \cap \bXx$ and $\bmu_\tau:\R_+\to\R^2$ 
by 
\be\label{eq:c_tau}
\bc_\tau(t) = (c_1^n,c_2^n), \quad
\bmu_\tau(t) = (\mu_1^n,\mu_2^n)
\qquad\text{for $t \in ((n-1)\tau, n\tau]$}.
\ee
 {Any pair $(\bc_\tau,\bmu_\tau)$ that has been obtained in this way from an iterated minimizer $\left(\bc^n\right)_{n\ge 1}$
will be referred to as \emph{$\tau$-approximate solution} below.}

\subsection{Weak solutions}
The goal of this section is to state our main result, that is the convergence of the JKO scheme. 
It requires the introduction of the notion of weak solution that will be obtained at the limit when the approximation parameter $\tau$ tends to 0.
\begin{Def}\label{def:weak}
  A pair $\left(\bc, \bmu\right)$ of phase concentrations $\bc=(c_1,c_2):\Omega\to[0,1]$ and phase potentials $\bmu=(\mu_1,\mu_2):\Omega\to\R$
  is said to be a {\em weak solution} to the problem~\eqref{eq:syst-PDE}, \eqref{eq:initial}, and~\eqref{eq:BC} if 
  \begin{itemize}
  \item \textbf{Regularity of concentrations:}
    $c_i\in L^\infty(\R_+;H^1(\O)) \cap L^2_{\rm loc}( {\R_+};H^2(\O)) \cap C(\R_+;L^2(\O))$; 
  \item \textbf{Regularity of potentials:}
    $\mu_i\in L^2_{\rm loc}(\R_+;L^{ {q_d}}(\O))$  {with $q_d=2$ if $d\in\{1,2\}$, and $q_d=\frac32$ if $d=3$}, 
    as well as $ {c_i} \grad \mu_i \in L^2(\R_+, L^2(\O))^d$, 
    and (normalization)
    $$
    \int_\O \left(c_1(\x,t) \mu_1(\x,t) + c_2(\x,t) \mu_2(\x,t)\right)\d\x = 0, \qquad \text{for a.e.}\; t \ge 0;
    $$
  \item \textbf{Initial and boundary conditions:}
    $c_i(\cdot,0) = c_i^0$ on $\O$, and $\grad c_i\cdot \n = 0$ on $\p\O\times \R_+$;
  \item \textbf{Volume constraint:}
    $c_1 + c_2 = 1$ a.e. in $\O \times \R_+$;
  \item \textbf{Continuity equations:}
    for all $\xi \in C^2(\ov \O)$ and all $t_1, t_2 \in \R_+$ with $t_2 \geq t_1$,
    \be\label{eq:weak1}
    \int_\O (c_i(\x,t_2) - c_i(\x, t_1)) \xi(\x) \d\x + m_i \int_{t_1}^{t_2} \int_\O \left(c_i \grad (\mu_i + \Psi_i) + \theta_i \grad c_i\right) \cdot \grad \xi \d\x \d t = 0;
    \ee
  \item \textbf{Steepest descent:}
    \eqref{eq:dmu} holds almost everywhere in $\O \times \R_+$, i.e.
    \be
    \mu_1  - \mu_2 = - \alpha \Delta c_1 + \chi (1-2c_1).
    \ee
  \end{itemize}
\end{Def}
Here is the convergence theorem for the minimization scheme. The  {global} existence of a weak solution is a by-product. 
 {All along the paper, $T$ denotes an arbitrary finite time horizon, and we make use of the shorten notation $Q_T$ for the space-time cylinder $\O\times(0,T)$.}
\begin{thm}[Convergence of the minimizing movement scheme]\label{thm:main} 
   {Assume that parameters $\alpha_1,\alpha_2,\chi>0$ and $\theta_1,\theta_2\ge0$ 
    as well as external potentials $\Psi_1,\Psi_2\in H^1(\O)$ are prescribed.
    Let $\bc^0=(c_1^0,c_2^0)$ be an initial condition satisfying \eqref{eq:c10+c20}.}

   {Then, for any sequence $\left(\tau_n\right)_{n\ge1} \subset (0,1)$ with $\tau_n\to0$
    and a corresponding sequence of $\tau_n$-approximate solutions $\left(\bc_{\tau_n}, \bmu_{\tau_n}\right)_{n\ge1}$,
    one can select a sub-sequence (not relabeled)}
  such that, 
  \begin{align*}
    & c_{i,\tau_n} \underset{n\to\infty}\longrightarrow c_i \quad \text{in the}\; L^\infty((0,T);H^1(\O))\text{-weak-$\star$ sense},\\
    &  \| c_{i,\tau_n}(\cdot, t) - c_i(\cdot, t)\|_{L^2(\O)}  \underset{n\to\infty}\longrightarrow 0 \quad \text{for all}\; t \in [0,T], \\
    & c_{i,\tau_n} \underset{n\to\infty}\longrightarrow c_i \quad \text{in}\; L^2((0,T);W^{1,d}(\O)),\\
    & c_{i,\tau_n} \underset{n\to\infty}\longrightarrow c_i \quad \text{weakly in}\; L^2((0,T);H^2(\O)),\\
    & \mu_{i,\tau_n} \underset{n\to\infty}\longrightarrow \mu_i \quad \text{for the weak topology of}\; L^2((0,T);L^{ {q_d}}(\O)),\\
    &  {c_{i,\tau_n} \grad \mu_{i,\tau_n}  \underset{n\to\infty}\longrightarrow  {{c_i}\,  \grad \mu_i} \quad \text{weakly in}\; L^2(Q_T)^d,} 
  \end{align*}
  and the limit $\left(\bc,\bmu\right)$ is a weak solution in the sense of Definition~\ref{def:weak}.
\end{thm}
 {We recall that the definition of $\tau$-approximate solutions involves
  the selection of particular minimizers in \eqref{eq:JKO} unless all densities have full support.}


\section{Proof of Theorem~\ref{thm:main}}\label{sec:proof}

We first establish some estimates on the approximate solution $\bc_\tau$. 
The very classical energy estimate and some straightforward consequences are 
exposed in Section~\ref{ssec:NRJ}. In Section~\ref{ssec:pos}, we show that the 
approximate solution remains bounded away from $0$ and $1$ if the thermal diffusion coefficients 
$\theta_i$ are positive.  {Then in Section~\ref{ssec:entro}, we make use of the flow interchange technique initially introduced in~\cite{MMS09}
to get enhanced regularity estimates on the approximate solutions}. 
The Euler-Lagrange equation are then obtained in Section~\ref{sec:Euler} thanks to a linearization  {technique} 
inspired from the work of Maury {\em et al.} \cite{MRS10}. The convergence of the approximate solution  {towards a weak solution} 
is finally established in Section~\ref{sec:conv}.

\subsection{Energy and distance estimate}\label{ssec:NRJ}

By definition of $\bc^n$ in \eqref{eq:JKO}, one has $\Ff_\tau^n(\bc^n)\le\Ff_\tau^n(\bc^{n-1})$, i.e.,
\be\label{eq:NRJ}
\Ee(\bc^{n}) + \frac{1}{2\tau} \bW^2(\bc^n, \bc^{n-1}) \leq \Ee(\bc^{n-1}), \qquad \forall n \ge 1.
\ee
Summing~\eqref{eq:NRJ} over $n$, and using that $\Ee(\bc) \ge \Ee_\star > -\infty$ for all $\bc \in \bAa$, 
we obtain the square distance estimate
\be\label{eq:square-distance}
\frac1\tau\sum_{n\ge1} \bW^2(\bc^n, \bc^{n-1}) \leq 2 \left(\Ee(\bc^0) - \Ee_\star\right) < +\infty.
\ee
This readily gives the approximate $1/2$-H\"older estimate 
\be\label{eq:Holder1}
\bW(\bc^{n_2},\bc^{n_1}) \leq C \sqrt{|n_2-n_1| \tau}, \qquad \forall n_1, n_2 \ge 0.
\ee
Bearing in mind the definition~\eqref{eq:c_tau} of the approximate solution $\bc_\tau$, we get that 
\be\label{eq:Holder2}
\bW(\bc_\tau(t),\bc_\tau(s)) \leq C\sqrt{|t-s|+\tau}, \qquad \forall s,t \ge 0.
\ee
We also deduce from the energy estimate~\eqref{eq:NRJ} that 
\be\label{eq:NRJ-decay}
\Ee(\bc^{n}) \leq \Ee(\bc^0) < \infty.
\ee
We deduce that 
\be\label{eq:LinfH1}
\int_\O |\grad c_1^n|^2 \d\x \leq \frac2\a \left(\Ee(\bc^0) + \sum_{i\in\{1,2\}} \|\Psi_i\|_{L^1}\right) < \infty, \qquad \forall n \ge 0.
\ee

\subsection{Positivity of the discrete solution in presence of thermal agitation}\label{ssec:pos}

The formula \eqref{eq:mu_i^n} suggests to give a proper sense to the quantity $\theta_i \log(c_i^n)$. 
This is the purpose of the following lemma, which is an adaptation to our framework of \cite[Lemma 8.6]{Santambrogio_OTAM}.
\begin{lem}\label{lem:log(c)}
  Let $\bc^n$ be a minimizer of $\Ff_\tau^n$ as in~\eqref{eq:JKO}. Assume that $\theta_i >0$, then $c_i^n >0$ a.e. in $\O$. 
  Moreover, $\theta_i \log(c_i^n) \in L^1(\O)$.
\end{lem}
\begin{proof}
With a slight abuse of notation, we denote by $\ov \bc = (\ov c_1, \ov c_2)$ the constant element of $\bXx\cap \bAa$ given by~\eqref{eq:m_i}.
  Fix $n$.
  Given $\eps \in (0,1)$, we introduce $\bc^\eps \in \bXx \cap \bAa$ by
  \be\label{eq:cieps}
  c_i^\eps = \eps \ov c_i + (1-\eps) c_i^n = c_i^n + \eps (\ov c_i - c_i^n).
  \ee
  Note that $c_i^\eps >0$ everywhere in $\O$,
  and that
  \begin{align}
    \label{eq:help001}
    c_i^\eps-\ov c_i = (1-\eps)(c_i^n-\ov c_i)
  \end{align}
  By optimality of $\bc^n$, the inequality 
  \begin{multline*}
    \Ee_{\rm therm}(\bc^n) - \Ee_{\rm therm}(\bc^\eps) \leq \Ee_{\rm Dir}(\bc^\eps) - \Ee_{\rm Dir}(\bc^n) 
    + \Ee_{\rm chem}(\bc^\eps) -  \Ee_{\rm chem}(\bc^n) \\
    + \Ee_{\rm ext}(\bc^\eps) - \Ee_{\rm ext}(\bc^n) 
    + \frac1{2\tau} \left( \bW^2(\bc^\eps, \bc^{n-1}) - \bW^2(\bc^n, \bc^{n-1}) \right)
  \end{multline*}
  holds for all $\eps \in (0,1)$.
  Then --- with generic constants $C$ that may change from line to line, but are independent of $\eps$ ---
  it follows directly from the definition of $\bc^\eps$ in \eqref{eq:cieps} that 
  $$
  \Ee_{\rm Dir}(\bc^\eps) \leq \Ee_{\rm Dir}(\bc^n), 
  $$
  that 
  $$
  \Ee_{\rm chem}(\bc^\eps) \leq  \Ee_{\rm chem}(\bc^n) + 
  \eps \int_\O ( \ov c_1 - c_1^n)(1-2 c_1^n) \d\x \leq  \Ee_{\rm chem}(\bc^n) + C \eps, 
  $$
  and that 
  $$
  \Ee_{\rm ext}(\bc^\eps) 
  =  \Ee_{\rm ext}(\bc^n) + \eps \sum_{i\in\{1,2\}}\left\langle c_i^n - \ov c_i \, , \, \Psi_i \right\rangle_{ {L^\infty(\O), L^1(\O)}} 
  \leq  \Ee_{\rm ext}(\bc^n)  + C \eps. 
  $$
  Moreover, the convexity of $\bW^2$ yields 
  $$
  \bW^2(\bc^\eps, \bc^{n-1}) 
  \leq \bW^2(\bc^n, \bc^{n-1}) + \eps \left( \bW^2(\ov \bc, \bc^{n-1}) -  \bW^2(\bc^n, \bc^{n-1})\right)
  \leq \bW^2(\bc^n, \bc^{n-1}) +C \eps.
  $$
  Combining the above inequalities, we obtain that 
  \be\label{eq:est_Ee_therm}
  \Ee_{\rm therm}(\bc^n) - \Ee_{\rm therm}(\bc^\eps) 
  =  \sum_{i\in\{1,2\}} \theta_i \int_\O (H(c_i^n) - H(c_i^\eps)) \d\x 
  \leq  {\frac{C\eps}{\tau}}. 
  \ee
  Let $i\in \{1,2\}$ be such that $\theta_i >0$, and denote by $A = \{ \x \in \O \, | \, c_i^n(\x) >0\}.$
  Then convexity of $H$ implies that 
  \be\label{eq:est_H_A}
  H(c_i^n) - H(c_i^\eps) \geq (c_i^n - c_i^\eps) \log(c_i^\eps) 
  = \eps (c_i^n - \ov c_i) \log(c_i^\eps) \quad \text{a.e. in $A$}.
  \ee
  On the other hand, since $c_i^\eps=\eps\ov c_i$ a.e. in $A^c$, 
  \be\label{eq:est_H_Ac}
  H(c_i^n) - H(c_i^\eps) = 1 - H(\eps \ov c_i) =  -\eps \ov c_i \log(\eps \ov c_i) + \eps \ov c_i \quad \text{a.e. in $A^c$}.
  \ee
  Integrating~\eqref{eq:est_H_A} over $A$ and \eqref{eq:est_H_Ac} over $A^c$ and using \eqref{eq:est_Ee_therm}, 
  we obtain  {after division by $\eps>0$}
  that 
  \be\label{eq:8.17}
  -\ov c_i \log(\eps \ov c_i) |A^c| + \int_A (c_i^n - \ov c_i) \log(c_i^\eps) \leq {\frac{C}{\tau}}. 
  \ee
  In view of \eqref{eq:help001}, we have
  \begin{align}
    \label{eq:help002}
    (c_i^n - \ov c_i) \log(c_i^\eps) \ge (c_i^n - \ov c_i) \log(\ov c_i),    
  \end{align}
  and therefore
  $$
  - \ov c_i \log(\eps \ov c_i) |A^c| + \int_A (c_i^n - \ov c_i) \log(\ov c_i) \leq  {\frac{C}{\tau}}. 
  $$
  Letting $\eps$ tend to $0$ produces a contradiction unless $|A^c| = 0$.
  Thus we have proved that $A = \O$ (up to a negligible set). 
  Moreover, in view of \eqref{eq:help002}, and since $(c_i^n - \ov c_i) \log(\ov c_i)\in L^1(\O)$, 
  Fatou's Lemma applies and leads to
  $$
  \int_\O  (c_i^n - \ov c_i) \log(c_i^n) \d\x  
  \leq \liminf_{\eps\to0} \int_\O  (c_i^n - \ov c_i) \log(c_i^\eps) \d\x 
  \leq  {\frac{C}{\tau}}. 
  $$
This latter inequality imposes that $\log(c_i^n)$ belongs to $L^1(\O)$  {for fixed $\tau>0$}.
\end{proof}
\subsection{Flow interchange and entropy estimate}\label{ssec:entro}

In the next lemma, our goal is to get an improved regularity estimate on $\bc$ by the mean of 
the flow interchange technique. 
\begin{lem}\label{lem:FI}
  There exists $C$ depending only on $\alpha, \chi, \Psi_i, \bc^0$, $m_i$, $T$, 
  such that 
  \be\label{eq:LapL2}
   {\sum_{i\in\{1,2\}} \theta_i \int_0^T \int_\O  |\grad \sqrt {c_{i,\tau}}|^2 \d\x \d t \; +}  
  \int_0^T \int_\O |\Delta \bc_\tau|^2 \d\x \d t \leq C, \qquad \forall T>0.
  \ee
  Since $\O$ is convex, it implies 
  \be\label{eq:L2H2}
  \|\bc_\tau\|_{L^2((0,T);H^2(\O))} \leq C.
  \ee
   {Moreover, there holds} 
  \be\label{eq:dcdn}
   {\grad c_i^n \cdot \n = 0 \quad \text{on }\;\p\O.}
  \ee
\end{lem}
\begin{proof}
   {Below, $s\ge0$ denotes an auxiliary time variable.}
  Let $\check c_i$ ($i\in \{1,2\}$) be the unique solution to 
  \be\label{eq:heat}
  \begin{cases}
    \p_s \check c_i = \Delta \check c_i &\text{in}\; \O \times (0,\infty), \\
    \grad \check c_i \cdot \n = 0 &\text{on}\; \p\O \times (0,\infty),\\
    (\check c_i)_{|_{s=0}} = c_i^n &\text{in}\; \O.
  \end{cases}
  \ee
  It is easy to check that $\int_\O \check c_i(\x,s) \d\x = \int_\O c_i^0(\x) \d\x$, 
  hence each $\check \bc(s)=(\check c_1(s), \check c_2(s))$ is an admissible competitor in~\eqref{eq:JKO},
  i.e., $\Ff_\tau^n(\bc^n)\le\Ff_\tau^n(\check\bc(s))$.
  After rearraging terms in $\Ff_\tau^n$ and dividing by $s>0$, 
  the passage to the limit $s\to0$ in this inequality produces
  \begin{align}
    \label{eq:help003}
    -\operatorname*{lim\,sup}_{s\to0}\frac{\d}{\d s}\Ee(\check\bc(s)) \le
    \frac{\d}{\d s}\bigg|_{s=0}\left(\frac1{2\tau}\bW^2(\check \bc(s), \bc^{n-1})\right).
  \end{align}
  To estimate the derivative on the right hand side above,
  we use that the heat equation~\eqref{eq:heat} is the gradient flow
  of the  Boltzmann entropy functional $\Hh(c) = \int_\O H(c) \d\x$, 
  which is displacement convex. 
  Therefore, solutions $\check\bc$ to \eqref{eq:heat} satisfy 
  the {\em Evolution Variational Inequality} \cite[Definition 4.5]{AG13_UGOT} 
  centered at $c_i^{n-1}$, 
  that is
  $$
  \frac12 \frac{\d}{\d s} W_i^2(\check c_i(s), c_i^{n-1}) \leq\frac{\Hh(c_i^{n-1}) -  \Hh(\check c_i(s))}{m_i}. 
  $$
  Division by $\tau$ and summation over $i$ leads to
  \be\label{eq:EVI}
  \frac{\d}{\d s}\left( \frac1{2\tau} \bW^2(\check \bc(s), \bc^{n-1})\right) 
  \leq \sum_{i\in\{1,2\}} \frac{\Hh(c_i^{n-1}) - \Hh(\check c_i(s))}{m_i\tau}, \qquad \forall s >0.
  \ee
  The estimate for the left-hand side of \eqref{eq:help003} is more combersome.
  We estimate the $s$-derivate of the various parts of $\Ee$ individually.
  To begin with, notice that there is no contribution from $\Ee_{\rm cons}$,
  since  the constraint $\check c_1(s)+\check c_2(s)=1$ is preserved by \eqref{eq:heat}.
  For working on the  other parts of $\Ee$, 
  we use that the solution $\check c_i$ of~\eqref{eq:heat} belongs to $C^\infty((0,\infty); H^2(\O))$, 
  is positive on $\O$ for $s>0$, and satisfies homogeneous Neumann boundary conditions.
  That makes $\check c_i$ regular enough to justify the following calculations,
  and in particular the integration by parts.
  For the Dirichlet energy, we obtain
  \be\label{eq:A_Dir}
  \frac{\d}{\d s}\Ee_{\rm Dir}(\check\bc(s)) 
  =  \frac{\d}{\d s}  \frac{\a}2 \int_\O |\grad \check c_1(s)|^2 \d\x 
  = -\a \int_\O \Delta \check c_1(s) \p_s \check c_1(s) \d\x 
  = -\a  \int_\O \left|\Delta \check c_1(s) \right|^2 \d\x.
  \ee
  The derivative of the chemical energy amounts to
  \begin{align*}
    \frac{\d}{\d s}\Ee_{\rm chem}(\check\bc(s))
    = &\; \frac{\d}{\d s}  \chi  \int_\O \check c_1(s) (1-\check c_1(s)) \d\x \\
    =&\; \chi \int_\O (1-2\check c_1(s)) \p_s \check c_1(s) \d\x
       = 2\chi \int_\O |\grad \check c_1(s)|^2 \d\x,
  \end{align*}
  and since $s\mapsto\Ee_{\rm Dir}(\check\bc(s))$ is non-increasing by the previous calculation,
  we can further conclude that
  \be\label{eq:A_chem}
    \frac{\d}{\d s}\Ee_{\rm chem}(\check\bc(s))
    \le 2 \chi \int_\O |\grad c_1^n|^2 \d\x \leq C,
    \ee
  where the second inequality follows from~\eqref{eq:LinfH1}.
  Next, for the thermal energy,
  \begin{align*}
    \frac{\d}{\d s}\Ee_{\rm therm}(\check\bc(s)) 
    =&  \sum_{i\in\{1,2\}} \theta_i \frac{\d}{\d s} \int_\O H(\check c_i(s)) \d\x 
       = \sum_{i\in\{1,2\}} \theta_i \int_\O \log(\check c_i(s)) \p_s \check c_i(s) \d\x \\
    =& \sum_{i\in\{1,2\}} \theta_i \int_\O \log(\check c_i(s)) \Delta \check c_i(s) \d\x 
       = -  \sum_{i\in\{1,2\}} \theta_i \int_\O \grad \log({\check c_i(s)}) \cdot \grad \check c_i(s) \d\x,
  \end{align*}
  and since $\check c_i(t)$ is smooth on $ {\ov\O}$ and bounded away from $0$, 
  we conclude that
  \be\label{eq:A_therm}
  \frac{\d}{\d s}\Ee_{\rm therm}(\check\bc(s)) 
  = - 4 \sum_{i\in\{1,2\}} \theta_i \int_\O |\grad \sqrt{\check c_i(s)} |^2 \d\x.
  \ee
  Finally, we obtain for the derivative of the potential
  \begin{align*}
    \frac{\d}{\d s}\Ee_{\rm ext}(\check \bc(s)) 
    =&\;  \frac{\d}{\d s} \sum_{i\in\{1,2\}}  \int_\O \Psi_i \check c_i(s) \d\x 
       = \int_\O (\Psi_1 - \Psi_2) \p_s \check c_1(s) \d\x  \\
     & =  \int_\O (\Psi_1 - \Psi_2) \Delta \check c_1(s) \d\x
       \leq \; \frac{\alpha}2 \|\Delta \check c_1(s)\|_{L^2}^2 + \frac1{2\alpha} \|\Psi_1 - \Psi_2\|_{L^2}^2.
  \end{align*}
  In combination with \eqref{eq:A_Dir}, this implies
  \be\label{eq:A_ext}
  \frac{\d}{\d s}\big(\Ee_{\rm Dir}(\check \bc(s)) + \Ee_{\rm ext}(\check \bc(s)) \big)
  \leq  - \frac{\alpha}2  \|\Delta \check c_1(s)\|_{L^2}^2 + C, \qquad \forall s >0.
  \ee
     {Since the initial condition $c_i^{n}$ in \eqref{eq:heat} belongs to $H^1(\O)$,
    we have that $\check c_i(s)\to c_i^{n}$ in $H^1(\O)$ as $s\to0$.} In particular, 
    $\Hh(\check c_i(s)) \to   \Hh(c_i^n)$ as $s\to0$.
  Now we substitute \eqref{eq:EVI} and \eqref{eq:A_chem}, \eqref{eq:A_therm}, \eqref{eq:A_ext} into \eqref{eq:help003},
  obtaining
  \begin{align}
    \label{eq:help004}
    \operatorname*{lim\,sup}_{s\to0}
    \left\{ \int_\O |\Delta \check c_1(s)|^2 \d\x 
     {\;+ 
    \sum_{i\in\{1,2\}} \theta_i \int_\O \left| \grad \sqrt{\check c_i(t)} \right|^2 \d\x}
    \right\}
    \leq C\left(1+ \sum_{i\in\{1,2\}}
    \frac{\Hh(c_i^{n-1}) - \Hh( {c_i^n})}{ {m_i}\tau}\right).
  \end{align}
   {Since the initial condition $c_i^{n}$ in \eqref{eq:heat} belongs to $H^1(\O)$,
    we have that $\check c_i(s)\to c_i^{n}$ in $H^1(\O)$ as $s\to0$.}
     {Convexity of the the functionals $c\mapsto\int|\Delta c|^2\d\x$ and $c\mapsto\int|\nabla\sqrt c|^2\d\x$
    implies lower semi-continuity with respect to convergence in $H^1(\O)$, and therefore
  }
  $$
  \int_\O |\Delta c_1^n|^2 \d\x+  \sum_{i\in\{1,2\}} \theta_i \int_\O \left|\grad \sqrt{c_i^n}\right|^2 \d\x\leq C\left(1+ \sum_{i\in\{1,2\}}
    \frac{\Hh(c_i^{n-1}) - \Hh(c_i^{n})}{m_i \tau}\right), \qquad \forall n \ge 1.
  $$
   {Multiplying by $\tau$ 
    and summing over $n \in \left\{1,\dots,\left\lceil \frac T\tau \right\rceil \right\}$ 
    leads to~\eqref{eq:LapL2};
    here we use that $0\le\sum_{i\in\{1,2\}}\Hh(c_i)\le2|\Omega|$ for all measurable $c_1,c_2:\O\to[0,1]$,  
    and that $|\Delta c_2^n| = |\Delta c_1^n|$.}

   {
    Concerning the boundary condition \eqref{eq:dcdn}:
    observe that the solution $\check c_i$ to \eqref{eq:heat} satisfies in particular $\nabla\check c_i(s)\cdot\n=0$ at each $s>0$.
    Since $\O$ is convex, one can show (see \cite[Chapter 3]{Grisvard85}) that 
    \[
    \int_\O\|\nabla^2\check c_i(s)\|^2\d\x\le\int_\O|\Delta\check c_i(s)|^2\d\x.
    \]
    In combination with the fact that $\check c_i(s)$ has values in $[0,1]$ only,
    it follows that the $L^2$-norm of $\Delta\check c_i(s)$ controls the full $H^2$-norm of $\check c_i(s)$, 
    i.e.,
    \begin{align}
      \label{eq:help005}
      \|\check c_i(s)\|_{H^2(\O)}^2\le C\left(1+\int_\O|\Delta\check c_i(s)|^2\d\x\right).
    \end{align}
    On the other hand, \eqref{eq:help004} implies in particular that $\Delta\check c_i(s)$ remains bounded in $L^2$ as $s\to0$.
    By Alaoglu, and since we already know that $\check c_i(s)$ tends to $c_i^n$ in $H^1(\O)$,
    it follows that $\check c_i(s)$ converges to $c_i^{n}$ weakly in $H^2(\Omega)$ as $s\to0$.
    Now, since the trace operator mapping $c\in H^2(\O)$ to $\nabla c\cdot\n\in L^2(\p\O)$ is weakly continuous,
    we conclude that $\nabla c_i^{n}\cdot\n=0$ as well.}

   {Finally, we obtain \eqref{eq:L2H2}:
    since we have just shown that $c_i^{n}$ satisfies homogeneous Neumann boundary conditions,
    the same argument as above shows that \eqref{eq:help005} holds also with $c_i^{n}$ in place of $\check c_i(s)$.
    Summation of \eqref{eq:LapL2} over $n$ directly produces \eqref{eq:L2H2}.}
\end{proof}
\begin{coro}
  \label{cor:JKO}
   {
    At each step $n\ge1$, there exists
    among the minimizers in \eqref{eq:JKO} at least one 
    which is a weak $H^2(\O)^2$-limit point for $\delta\to0$ 
    of the set of minimizers of the functionals $\Ff_\tau^{n,\delta}$ defined in \eqref{eq:Ffdelta}.}
\end{coro}
\begin{proof}
  Since $\bc^{n-1,\delta}\to\bc^{n-1}$ in $L^1(\O)^2$ as $\delta\to0$,
  the $\Ff_\tau^{n,\delta}$ $\Gamma$-converge to $\Ff_\tau^n$, for instance in the topology induced by $\bW$ on $\bAa$.
  In complete analogy to the previous proof, we obtain for each minimizer $\bc^{n,\delta}$ of $\Ff_\tau^{n,\delta}$ that
  \begin{align*}
    \|c^{n,\delta}\|_{H^2}^2 \le C\left(1+\sum_{i\in\{1,2\}}\frac{\Hh(c_i^{n-1,\delta})-\Hh(c_i^{n,\delta})}{m_i\tau}\right) \le \frac C\tau.
  \end{align*}
  Therefore, there is a sequence $\bc^{n,\delta_k}$ that converges weakly in $H^2(\O)^2$ to a limit $\bc^n$
  that minimizes $\Ff_\tau^n$.
\end{proof}

\subsection{Euler Lagrange equations}\label{sec:Euler}

The goal of this section is to characterize the minimizer $\bc^n$ of~\eqref{eq:JKO}. 
 {We have anticipated the main result already in Lemma \ref{lem:preEL}.
  It is a consequence of the following linearized optimality condition for $\bc^n$.}
\begin{lem}\label{lem:linearized}
  In each step $n\ge1$,
  there exist Kantorovich potentials $\bvarphi^n= (\varphi_1^n, \varphi_2^n)$ 
  such that the quantity $F^n$ given in \eqref{eq:F^n} satisfies
  the following linearized optimality condition:
  \be\label{eq:linearized}
  \int_\O F^n c_1^n \d\x \leq 
  \int_\O F^n c_1 \d\x, \qquad \forall \bc = (c_1,c_2) \in \bAa.
  \ee
\end{lem}
\begin{proof}
  The proof is inspired from~\cite[Lemma 3.1]{MRS10}, see also~\cite[Lemma 3.2]{CGM17}. 
  Assume first that $c_i^{n-1}>0$ almost everywhere in $\O$, so that 
  the Kantorovich potentials $\varphi_1^n,\psi_1^n$ from $c_1^n$ to $c_1^{n-1}$ are unique  {up to addition of a global constant},
  that we fix at an arbitrary value for this proof.
  For any given $\bc\in\bAa \cap \bXx$, we perturb $\bc^n$ into 
  \be\label{eq:c^eps}
  \bc^\eps = (1-\eps) \bc^n + \eps \bc, \qquad \forall \eps \in (0,1).
  \ee
  Clearly, $\bc^\eps$ belongs to $\bAa \cap \bXx$. 
  Defining $(\varphi^\eps_i,\psi_i^\eps)$ the Kantorovich potential from $c_i^\eps$ to $c_i^{n-1}$, we infer from~\eqref{eq:DP}
  that 
  \[
    \begin{cases}
      \ds \frac{1}{2}W_i^2(c_i^\eps,c_i^{n-1})=\int_\O \varphi_i^\eps(\x) c_i^\eps(\x) \d\x +\int_\O \psi_i^\eps(\by)c_i^{n-1}(\by)\d\by,\\
      \ds \frac{1}{2}W_i^2(c_i^n, c_i^{n-1}) \geq\int_\O \varphi_i^\eps(\x) c_i^n(\x) \d\x +\int_\O \psi_i^\eps(\by)c_i^{n-1}(\by)\d\by.
    \end{cases}
  \]
  Subtracting the two above relations and using the definition~\eqref{eq:c^eps} of $\bc^\eps$, one gets 
  $$
  \frac{1}{2\tau}\left( W_i^2(c_i^\eps,c_i^{n-1}) - W_i^2(c_i^n, c_i^{n-1}) \right) \leq \eps \int_\O \frac{\varphi_i^\eps}{\tau}\left( c_i- c_i^n\right) \d\x.
  $$
  Hence, using $\bc^n, \bc \in \bXx$, one gets that 
  \be\label{eq:W_eps}
  \frac{1}{2\tau}\left( \bW^2(\bc^\eps,\bc^{n-1}) - \bW^2(\bc^n, \bc^{n-1}) \right) \leq 
  \eps \int_\O \left(\frac{\varphi_1^\eps}{\tau} -\frac{\varphi_2^\eps}{\tau} \right)\left( c_1- c_1^n\right) \d\x.
  \ee
  On the other hand, the convexity of $\Ee_{\rm Dir}$ and $\Ee_{\rm therm}$, the linearity of $\Ee_{\rm ext}$, and the concavity of $\Ee_{\rm chem}$
  yield
  \begin{multline}\label{eq:Ee_eps}
    \Ee(\bc^\eps) - \Ee(\bc^n) \leq
    \eps \int_\O \alpha \grad c_1^\eps \cdot \grad (c_1 - c_1^n)  \d\x + \eps \int_\O \chi (1-2c_1^n)(c_1 - c_1^n) \d\x \\
    + \eps \int_\O  f(c_1^\eps) (c_1 - c_1^n) \d\x 
    +  \eps \int_\O \left( \Psi_1 - \Psi_2\right) (c_1 - c_1^n) \d\x.
  \end{multline}
  Bearing in mind that $c_1^n$ is a minimizer, the combination of~\eqref{eq:W_eps} with \eqref{eq:Ee_eps} leads to 
  \begin{multline*}
    0 \leq \frac{\Ff_\tau^n(\bc^\eps) - \Ff_\tau^n(\bc^n)}\eps 
    \leq  \int_\O \left(\frac{\varphi_1^\eps}{\tau} -\frac{\varphi_2^\eps}{\tau} \right)\left( c_1- c_1^n\right) \d\x + 
    \int_\O \alpha \grad c_1^\eps \cdot \grad (c_1 - c_1^n)  \d\x \\
    + \int_\O \chi (1-2c_1^n)(c_1 - c_1^n) \d\x +  \int_\O  f(c_1^\eps) (c_1 - c_1^n) \d\x 
    + \int_\O \left( \Psi_1 - \Psi_2\right) (c_1 - c_1^n) \d\x, \quad \forall \eps >0.
  \end{multline*}
  We can consider the limit $\eps \to 0$ in the right-hand side of the above expression.
  From the definition~\eqref{eq:c^eps} of $\bc^\eps$, it is clear that $c_1^\eps$ converges in 
  $H^1(\O)$ towards $c_1^n$  and that $f(c_1^\eps)$ converges in $L^1(\O)$ towards $f(c_1^n)$, 
  while $\varphi_i^\eps$ converges uniformly towards $\varphi_i^n$ (see for instance~\cite[Theorem~1.52]{Santambrogio_OTAM}),
  so that \eqref{eq:linearized} holds thanks to Lemma~\ref{lem:FI}.

  Assume now that $c_i^{n-1} = 0$ on some part of $\O$. 
   {By definition of $\bc^n$ in and after \eqref{eq:JKO}, there exists a sequence of minimizers $\bc^{n,\delta}$
    of the respective functionals $\Ff_\tau^{n,\delta}$ given in \eqref{eq:Ffdelta} that converge weakly in $H^2(\O)^2$ to $\bc^n$.
    Defining $F^{n,\delta}$ as $F^n$ in~\eqref{eq:F^n} upon replacing $c_i^{n}$ by $c_i^{n,\delta}$ 
    and $\varphi_i^n$ by the Kantorovich potential $\varphi_i^{n,\delta}$ sending $c_i^{n,\delta}$ to $c_i^{n-1,\delta}$, 
    then we get from the reasoning above that }
  \be\label{eq:F^n_delta}
  \int_\O F^{n,\delta} (c_1 - c_1^{n,\delta}) \ge 0, \qquad \forall \bc=(c_1,c_2) \in\bAa\cap \bXx.
  \ee
   {The weak convergence of the $c^{n,\delta}_i$ to $c^n_i$ in $H^2(\O)$,
    and the induced uniform convergence of the Kantorovich potentials $\varphi_i^{n,\delta}$ to Kantorovich potentials $\varphi_i^n$ 
    sending $c_i^n$ to $c_i^{n-1}$ (cf. \cite[Theorem 1.52]{Santambrogio_OTAM})
    is sufficient to deduce \eqref{eq:linearized} from \eqref{eq:F^n_delta} in the limit $\delta\to0$.}
\end{proof}
We are now in the position to prove Lemma \ref{lem:preEL} 
which has been essential for the definition of the potentials $\bmu$.
\begin{proof}[Proof of Lemma \ref{lem:preEL}]
   {Since $0 \leq c_1^n \leq 1$, the minimization problem \eqref{eq:linearized} can be solved thanks to 
    the bathtub principle \cite[Theorem 1.14]{LL01}. It amounts to saturate the sublevel sets of $F^n$ with $c_1^n=1$ until 
    all the mass $\int_\O c_i^0$ has been allocated. This} implies the existence of some $\ell \in \R$ such that 
  \be\label{eq:bathtub}
  c_1^n(\x) = \begin{cases}
    0 & \text{if}\; F^n(\x) > \ell, \\
    1 & \text{if}\; F^n(\x) < \ell,
  \end{cases}
  \qquad 
  c_2^n(\x) = \begin{cases}
    1 & \text{if}\; F^n(\x) > \ell, \\
    0 & \text{if}\; F^n(\x) < \ell.
  \end{cases}
  \ee
   {Given $F^n$, the solution $c_1^n$ to the minimization problem~\eqref{eq:linearized} is in general not unique 
    since a prescribed amount of mass can be distributed on the level set $\{F^n = \ell\}$ in different ways without changing 
    the value of the functional. But this lack of uniqueness does not affect the relations~\eqref{eq:bathtub}.}

   {Recall that we still have the freedom 
    to change the Kantorovich potential $\varphi_1^n$ in the definition \eqref{eq:F^n} of $F^n$ 
    by addition of a global constant.
    We choose that constant to enforce $\ell=0$ in \eqref{eq:bathtub} above.
    With that normalization, \eqref{eq:bathtub} implies \eqref{eq:EL}.}
\end{proof}
\begin{lem}
  \label{lem:postEL}
  The chemical potentials $\bmu$ satisfy the following $\tau$-uniform estimates:
  \be\label{eq:est-mu_tau}
  \|\mu_{i,\tau}\|_{L^2((0,T);L^{ {q_d}}(\O))} \leq C, 
  \qquad 
  \|\mu_{1,\tau}-\mu_{2,\tau}\|_{L^2(Q_T)} \leq C, 
  \qquad 
  \left\|\sqrt{c_{i,\tau}} \grad \mu_{i,\tau} \right\|_{L^2(Q_T)} \leq C. 
  \ee
\end{lem}
\begin{proof}
  Recall the definition of $\bmu$ in \eqref{eq:defmu}.
  It follows from \eqref{eq:LapL2} that 
  \be\label{eq:diff-muL2}
  \iint_{Q_T} \left(\mu_{1,\tau} -\mu_{2,\tau}\right)^2 \d\x \d t\leq C. 
  \ee
  The definition of $\ov \mu^n$ implies that 
  \begin{align*}
    \grad \ov \mu^n =&  \sum_{i\in\{1,2\}} c_i^n \grad \mu_i^n + \grad c_1^n \left(  \mu_1^n - \mu_2^n \right) \\
    = & -  \sum_{i\in\{1,2\}}  \left( c_i^n \frac{\grad \varphi_i^n}{\tau} - c_i^n \grad \Psi_i - \theta_i \grad c_i^n \right) + \grad c_1^n \left(  \mu_1^n - \mu_2^n \right) .
  \end{align*}
  Using the triangle inequality,  Cauchy-Schwarz inequality, and $0 \leq c_i^n \leq 1$, we get
  \begin{multline*}
    \|\grad \ov \mu^n \|_{L^{1}(\O)}  \leq \frac1\tau\sum_{i\in\{1,2\}} \left(\int_\O c_i^n\right)^{1/2} \left(\int_\O c_i^n |\grad \varphi_i^n|^2 \d\x \right)^{1/2} \\
    + \sum_{i\in\{1,2\}} \|\grad \Psi_i\|_{L^1(\O) {^d}} + {\|\grad c_1^n\|}_{L^2(\O)^d} \left( |\O|^{1/2}(\theta_1+\theta_2) +   \|\mu_1^n-\mu_2^n\|_{L^2(\O)} \right).
  \end{multline*}
  Since $\bc^n \in \bAa$, one has $\int_\O c_i^n \d\x = \int_\O c_i^0 \d\x$. Thus it follows from~\eqref{eq:Kanto-grad}, \eqref{eq:square-distance}, \eqref{eq:LinfH1} and \eqref{eq:diff-muL2} that 
  \be
   {\|\grad \ov \mu_\tau\|^2_{L^2\left((0,T);L^1(\O)^d\right)} \leq} \sum_{n=0}^{\lceil \frac{T}{\tau} \rceil} \tau \|\grad \ov \mu^n \|_{L^{1}(\O) {^d}}^2 \leq C.
  \ee
  Bearing~\eqref{eq:ov-mu_mean} in mind, we can use the Poincar\'e-Sobolev inequality and get that 
  \be
  \|\ov \mu_\tau \|_{L^2((0,T);L^{d/(d-1)}(\O))} \leq C.
  \ee
  Since 
  $$
  \mu_1^n = \ov \mu^n + c_2^n (\mu_1^n - \mu_2^n), \qquad 
  \mu_2^n = \ov \mu^n + c_1^n (\mu_2^n - \mu_1^n),
  $$
  we deduce from~\eqref{eq:diff-muL2} the desired $L^2((0,T);L^{ {q_d}}(\O))$ estimates on the phase potentials $\mu_i^n$  {with $q_d$ defined in Definition \ref{def:weak}}.
  Finally, the combination of the relations~\eqref{eq:mu_i^n}, \eqref{eq:Kanto-grad},  {\eqref{eq:LapL2}} and \eqref{eq:square-distance} yields
  $$
  \iint_{Q_T} c_{i,\tau} \left|\grad \mu_{i,\tau} \right|^2 \d\x\d t \leq C. 
  $$
\end{proof}

The following lemma is a first step towards the recovery of the weak formulation~\eqref{eq:weak1}. 
\begin{lem}\label{lem:weak-disc}
For any $\xi \in C^2(\ov \O)$, there holds 
\be\label{eq:Otto-calculus}
\left|\int_\O \left(c_i^{n} - c_i^{n-1}\right) \xi \d\x + \tau m_i \int_\O \left( c_i^n \grad \left( \mu_i^n + \Psi_i\right)  + \theta_i \grad c_i^n\right)   \cdot \grad \xi \d\x \right| \leq  \frac12 W_i^2(c_i^{n}, c_i^{n-1}) \|D^2\xi\|_\infty.
\ee
\end{lem}
\begin{proof}
The optimal transport map
$$\bt^n_i(\x) = \x - m_i\grad \varphi_i^n(\x), \qquad \forall \x \in \O$$
sending $c_i^n$ to $c_i^{n-1}$ maps $\O$ into itself because $\O$ is convex.
Therefore, since $c_i^{n-1} = \bt_i^n\# c_i^n$ and thanks to~\eqref{eq:mu_i^n}, one gets that 
\begin{multline*}
\int_\O \left(c_i^{n} - c_i^{n-1}\right) \xi \d\x + \tau m_i \int_\O \left( c_i^n \grad \left( \mu_i^n + \Psi_i\right)  + \theta_i \grad c_i^n\right)   \cdot \grad \xi \d\x \\
= \int_\O  \left(\xi(\x) - \xi(\bt_i^n(\x)) - m_i \grad \xi(\x) \cdot \grad \varphi_i^n(\x)\right) c_i^n(\x)  \d\x
\end{multline*}
for all $\xi \in C^2(\ov \O)$. The Taylor expansion of $\xi$ at point $\x$ provides
$$
\left| \xi(\bt^n_i(\x)) - \xi(\x) + m_i \grad \xi(\x) \cdot \grad \varphi_i(\x) \right| \leq \frac{1}2 \|D^2\xi\|_\infty |\bt_i^n(\x)-\x|^2, \qquad \forall \x \in \O, 
$$
so that 
\begin{multline*}
\left| \int_\O (c_i^n - c_i^{n-1}) \xi \d\x +
 \tau m_i \int_\O  \left( c_i^n \grad \left( \mu_i^n + \Psi_i\right)  + \theta_i \grad c_i^n\right) \cdot \grad \xi \d\x \right|  \\
\leq \frac12 \|D^2\xi\|_\infty  \int_\O |\bt_i^n(\x)-\x|^2 c_i^n(\x) \d\x,
\end{multline*}
which is exactly the desired result. 
\end{proof}
\subsection{Convergence towards a weak solution}\label{sec:conv}

The goal of this section is to consider the limit $\tau \to 0$. This requires some compactness on the approximate phase field $\bc_\tau$ and 
on the approximate potential $\bmu_\tau$. 
In what follows, $\bAa$ is equipped with the topology corresponding to the distance $\bW$.
\begin{prop}\label{lem:compact}
There exist $\bc \in C([0,T];L^2(\O)) \cap L^2((0,T);H^2(\O))\cap L^\infty((0,T); H^1(\O))$ with $\bc(t) \in \bAa \cap \bXx$ for a.e. $t\in [0,T]$, and 
$\bmu \in L^2((0,T);L^{ {q_d}}(\O))$ such that, up to the extraction of a subsequence, the following 
convergence properties hold:
\begin{subequations}
\begin{align}
& c_{i,\tau} \underset{\tau\to0}\longrightarrow c_i \quad \text{in the}\; L^\infty((0,T);H^1(\O))\text{-weak-$\star$ sense},\\
& \| c_{i,\tau}(\cdot, t) - c_i(\cdot, t)\|_{L^2(\O)}  \underset{\tau\to0}\longrightarrow 0 \quad \text{for all}\;t\in [0,T],
\label{eq:unif-L2}\\
& c_{i,\tau} \underset{\tau\to0}\longrightarrow c_i \quad \text{in}\; L^2((0,T);W^{1,d}(\O)),\\
& c_{i,\tau} \underset{\tau\to0}\longrightarrow c_i \quad \text{weakly in}\; L^2((0,T);H^2(\O)),\\
& \mu_{i,\tau} \underset{\tau\to0}\longrightarrow \mu_i \quad \text{for the weak topology of}\; L^2((0,T);L^{ {q_d}}(\O)),\\
& c_{i,\tau} \grad \mu_{i,\tau}  \underset{\tau\to0}\longrightarrow c_i \grad \mu_i \quad \text{weakly in}\; L^2(Q_T)^d. \label{eq:c-grad-mu}
\end{align}
\end{subequations}
\end{prop}
\begin{proof}
All the convergence properties stated below occur up to the extraction of a subsequence when $\tau$ tends to $0$. 
We deduce from Estimate~\eqref{eq:LinfH1} that the family $\left(\bc_\tau\right)_{\tau>0}$ is bounded in $L^\infty((0,T);H^1(\O))$. 
Hence we can assume that 
$\bc_\tau$ tends to some $\bc$ in the $L^\infty((0,T);H^1(\O))$-weak-$\star$ sense. Moreover, since $0 \leq c_{i,\tau} \leq 1$, we also have that 
$0 \leq c_i \leq 1$ a.e. in $Q_T$. 
We also infer from Estimate~\eqref{eq:L2H2} that $\bc_\tau$ converges weakly in $L^2((0,T);H^2(\O))$ towards $\bc$.

As a consequence of the $L^\infty$ bound on $c_{i,\tau}$ and of the Benamou-Brenier formula, we get that 
$$
\|c_i^{(1)} - c_i^{(2)}\|_{H^1(\O)'} \leq \frac1{m_i} W_i (c_i^{(1)}, c_i^{(2)}), \qquad \forall c_i^{(1)}, c_i^{(2)} \in \Aa_i,
$$
(see more precisely \cite[Lemma 3.4]{Santambrogio_OTAM}). Therefore, we infer from \eqref{eq:Holder2} that 
$$
\|c_{i,\tau}(t) - c_{i,\tau}(s)\|_{H^1(\O)'} \leq C \sqrt{|t-s|+\tau}, \qquad \forall s,t \in [0,T].
$$
Let $\dt>0$ and let $t \in [0,T-\dt]$, then 
$$
\|c_{i,\tau}(t + \dt) - c_{i,\tau}(t) \|_{L^2(\O)}^2 \, {\leq}\, \|c_{i,\tau}(t + \dt) - c_{i,\tau}(t) \|_{H^1}\|c_{i,\tau}(t + \dt) - c_{i,\tau}(t) \|_{(H^1)'}
\leq C \sqrt{\dt+\tau}.
$$
Bearing in mind the $L^\infty((0,T);H^1(\O))$ estimate on $c_{i,\tau}$, we can apply a refined version of the Arzel\`a-Ascoli theorem 
\cite[Prop. 3.3.1]{AGS08} 
to obtain that $c_i \in C([0,T];L^2(\O))$ and that 
$$
c_{i,\tau}(t)  \underset{\tau\to0}\longrightarrow c_i(t) \quad \text{in}\; L^2(\O) \; \text{for all}\; t \in [0,T].
$$
Together with the estimate $0 \leq c_{i,\tau} \leq 1$, we deduce that $\bc \in C([0,T];L^p(\O))^2$ and that 
$$
\bc_{\tau}  \underset{\tau\to0}\longrightarrow \bc \quad \text{in}\; L^p(Q_T)^2, \quad \forall p \in [1,+\infty).
$$
This implies in particular some strong convergence in $L^2(Q_T)$, which can be combined with the weak convergence 
in $L^2((0,T);H^2(\O))$ thanks to interpolation arguments to derive some strong convergence in $L^2((0,T);H^s(\O))$ for 
any $s < 2$. The continuous embedding of $H^s(\O)$ into $W^{1,d}(\O)$ when $s \geq1+\max(0,\frac{d-2}2)$ (see for instance \cite[Theorem 6.7]{NPV12}) 
ensures that 
\be\label{eq:conv-W1d}
\bc_\tau\underset{\tau\to0}\longrightarrow \bc \quad \text{in}\; L^2((0,T);W^{1,d}(\O))^2.
\ee

Let us switch to the phase potentials $\bmu$.
Thanks to Lemma~\ref{lem:postEL}, we have (uniform w.r.t. $\tau$) $L^2((0,T);L^{ {q_d}}(\O))$ estimates on $\mu_{i,\tau}$. Hence
there exists $\bmu$ in $L^2((0,T);L^{ {q_d}}(\O))^2$ such that 
\be\label{eq:conv-mu_i}
\mu_{i,\tau} \underset{\tau\to0}\longrightarrow \mu_i \quad \text{weakly in}\; L^2((0,T);L^{ {q_d}}(\O)).
\ee
In Lemma~\ref{lem:postEL}, we also established a (uniform w.r.t. $\tau$) $L^2(Q_T)^d$ estimate on $\left(\sqrt{c_{i,\tau}} \grad \mu_{i,\tau} \right)_{\tau>0}$.
Since $0 \leq c_{i,\tau}\leq 1$, it implies a uniform $L^2(Q_T)^d$ estimate on $\left({c_{i,\tau}} \grad \mu_{i,\tau} \right)_{\tau>0}$.
Therefore, there exists $\boldsymbol\vartheta_i \in L^2(Q_T)^d$ such that ${c_{i,\tau}} \grad \mu_{i,\tau}$ converges weakly in 
$L^2(Q_T)^d$ to $\boldsymbol\vartheta_i$ as $\tau$ tends to 0. It remains to show that $\boldsymbol\vartheta_i = c_i \grad \mu_i$. 
First, the distributions $c_i$ and $\grad \mu_i$ can be multiplied since $c_i$ belongs to $L^2((0,T);W^{1,d}(\O))$ and $\grad \mu_i$ belongs 
to $L^2((0,T);W^{-1, {q_d}}(\O))$  {with $q_d\leq \frac{d}{d-1}$}. Moreover, for all $\bphi \in C^\infty_c(Q_T)^d$, one has 
\be\label{eq:poipoi_1}
\iint_{Q_T} c_{i,\tau} \grad \mu_{i,\tau} \cdot \bphi \d\x \d t = - \iint_{Q_T} \mu_{i,\tau} \left( \grad c_{i,\tau} \cdot \bphi + c_{i,\tau} \div \bphi \right) \d\x \d t.
\ee
Thanks to~\eqref{eq:conv-W1d} and \eqref{eq:conv-mu_i}, we can pass in the limit in the right-hand side of the above expression. This leads to 
\be\label{eq:poipoi_2}
\iint_{Q_T} c_{i,\tau} \grad \mu_{i,\tau} \cdot \bphi \d\x \d t \underset{\tau\to0} \longrightarrow 
- \iint_{Q_T} \mu_{i} \left( \grad c_{i} \cdot \bphi + c_{i} \div \bphi \right) \d\x \d t = \langle c_{i} \grad \mu_{i} \; , \; \bphi \rangle_{\Dd', \Dd}.
\ee
As a consequence, $\boldsymbol \vartheta_i =  c_{i} \grad \mu_{i}$ in the distributional sense, thus also in $L^2(Q_T)$.

Since $\bc_\tau(t)$ converges in $L^2(\O)^2$ towards $\bc(t)$ for all $t\in[0,1]$, there holds 
$$
\int_\O c_i(t) \d\x = \int_\O c_i^0 \d\x, \qquad c_1(\x,t) + c_2(\x,t) = 1 \quad \text{for all} \; t \in [0,T].
$$
Moreover, since $\bc(t)$ belongs to $H^1(\O)$ for a.e. $t \in (0,T)$, $\bc(t)$ belongs to $\bAa \cap \bXx$ for a.e. $t \in (0,T)$.
This concludes the proof of Proposition~\ref{lem:compact}.
\end{proof}

We have all the necessary convergence properties to pass to the limit $\tau\to 0$ and to identify the limit $(\bc, \bmu)$ exhibited in Proposition~\ref{lem:compact} 
as a weak solution in the sense of Definition~\ref{def:weak}. 

\begin{prop}\label{prop:identify}
Let $(\bc, \bmu)$ be as in Proposition~\ref{lem:compact}, then $(\bc, \bmu)$ is a weak solution to the problem~\eqref{eq:syst-PDE}, \eqref{eq:initial}--\eqref{eq:BC} 
in the sense of Definition~\ref{def:weak}.
\end{prop}
\begin{proof}
Since $c_{i,\tau}$ and $\mu_{i,\tau}$ tend weakly  in $L^2((0,T);H^2(\O))$ and $L^2((0,T);L^{ {q_d}}(\O))$ towards $c_{i}$ and $\mu_i$ respectively, 
we can pass to the limit in~\eqref{eq:diff-mu}. Moreover, one can also pass to the limit in the relation $0 = \grad c_{1,\tau}\cdot \n$ established in Lemma~\ref{lem:FI}, 
leading to $\grad c_1 \cdot \n = 0$ on $\p\O$. 

It only remains to recover the weak formulation~\eqref{eq:weak1}. Let $t_1, t_2 \in [0,T]$ with $t_2 \ge t_1$, then summing~\eqref{eq:Otto-calculus} over 
$n \in \left\{\left\lceil \frac{t_1}\tau \right\rceil + 1, \dots, \left\lceil \frac{t_2}\tau \right\rceil \right\}$ yields 
\begin{multline*}
\left|\int_\O \left(c_{i,\tau}(t_2) - c_{i,\tau}(t_1) \right) \xi \d\x + m_i \int_{\left\lceil \frac{t_1}\tau \right\rceil \tau}^{\left\lceil \frac{t_2}\tau \right\rceil \tau} \int_\O 
\left(c_{i,\tau} \grad \left(\mu_{i,\tau}+ \Psi_i\right) + \theta_i \grad c_{i,\tau}\right)\cdot \grad \xi\d\x \d t \right| \\
 \leq \frac12 \|D^2\xi\|_\infty \sum_{n=\left\lceil \frac{t_1}\tau \right\rceil}^{\left\lceil \frac{t_2}\tau \right\rceil} 
W_i^2(c_i^n, c_i^{n-1}) \leq C \tau,
\end{multline*}
the last inequality being the consequence of the squared distance estimate~\eqref{eq:square-distance}.
We can pass to the limit $\tau \to 0$ in the above relation thanks to Proposition~\ref{lem:compact}. 
\end{proof}
We have finally proved Theorem~\ref{thm:main} that is a 
combination of Propositions~\ref{lem:compact} and~\ref{prop:identify}.

\subsection*{Acknowledgements}  
 {The authors warmly thank the referees for their valuable remarks and suggestions.}
This research was supported by the DFG Collaborative Research Center TRR 109, ``Discretization in Geometry and Dynamics" and by 
the French National Research Agency (ANR) through grants ANR-13-JS01-0007-01 (project GEOPOR) and ANR-11-LABX-0007-01 (Labex CEMPI).


\end{document}